\documentclass{amsart}
\usepackage{amsmath,amssymb, amsthm}
\usepackage[english]{babel}

\usepackage{ifthen,srcltx}
\usepackage{hyperref}
\usepackage{url}
\usepackage{enumerate}

\usepackage{algorithm}
\usepackage[noend]{algpseudocode}

\newtheorem{theorem}{Theorem}[section]
\newtheorem*{theorem*}{Theorem}
\newtheorem{lemma}[theorem]{Lemma}
\newtheorem{prop}[theorem]{Proposition}

\theoremstyle{definition}
\newtheorem{definition}{Definition}
\newtheorem{question}{Question}
\newtheorem*{remark}{Remark}

\newtheorem*{definition*}{Definition}
\newtheorem*{corollary*}{Corollary}

\newtheorem{theostar}{Theorem}

\newcommand{\showcomments}{yes}
\renewcommand{\showcomments}{no}

\newsavebox{\commentbox}
{\ifthenelse{\equal{\showcomments}{yes}}%
{\footnotemark
        \begin{lrbox}{\commentbox}
        \begin{minipage}[t]{1.25in}\raggedright\sffamily\tiny
        \footnotemark[\arabic{footnote}]}
{\begin{lrbox}{\commentbox}}}
{\ifthenelse{\equal{\showcomments}{yes}}
{\end{minipage}\end{lrbox}\marginpar{\usebox{\commentbox}}}
{\end{lrbox}}}

\newcommand{\eps}{\varepsilon}

\newcommand{\wdt}[1]{\widetilde #1}

\newcommand{\ovl}[1]{\overline #1}

\newcommand{\pl}{\partial}
\newcommand{\bs}{\backslash}

\newcommand{\vol}{\operatorname{vol}}

\newcommand{\tr}{\operatorname{tr}}

\newcommand{\aut}{\operatorname{Aut}}

\newcommand{\fix}{\mathrm{Fix}}
\newcommand{\aff}{\operatorname{Aff}}

\newcommand{\PSL}{\mathrm{PSL}}
\newcommand{\SL}{\mathrm{SL}}
\newcommand{\SU}{\mathrm{SU}}
\newcommand{\SO}{\mathrm{SO}}
\newcommand{\OO}{\mathrm{O}}
\newcommand{\GL}{\mathrm{GL}}

\newcommand{\A}{\mathrm{A}}

\newcommand{\Sp}{\mathrm{Sp}}
\renewcommand{\H}{\mathrm{H}}

\newcommand{\sol}{\mathbf{Sol}}

\newcommand{\frH}{\mathfrak{H}}
\newcommand{\frN}{\mathfrak{N}}

\newcommand{\ad}{\mathrm{Ad}}

\newcommand{\abs}{\mathrm{abs}}
\newcommand{\rel}{\mathrm{rel}}
\newcommand{\res}{\mathrm{res}}

\newcommand{\isom}{\mathrm{Isom}}

\newcommand{\TT}{\mathbb T}

\newcommand{\CC}{\mathbb C}
\newcommand{\RR}{\mathbb R}
\newcommand{\ZZ}{\mathbb Z}
\newcommand{\HH}{\mathbb H}
\newcommand{\PP}{\mathbb P}
\newcommand{\QQ}{\mathbb Q}

\DeclareMathOperator{\ind}{ind}
\DeclareMathOperator{\sh}{sh}
\DeclareMathOperator{\arcosh}{arcosh}
\DeclareMathOperator{\diag}{diag}
\newcommand{\normal}{\trianglelefteq}
\DeclareMathOperator{\gal}{Gal}

\newcommand{\hyp}{\nobreakdash-\hspace{0pt}}
\newcommand{\3}[1]{3\hyp}

\providecommand{\Q}{{\mathbb Q}}
\providecommand{\R}{{\mathbb R}}
\providecommand{\C}{{\mathbb C}}
\providecommand{\Z}{{\mathbb Z}}

\newcommand{\cO}{{\mathcal O}}
\newcommand{\cQ}{{\mathcal Q}}
\newcommand{\cT}{{\mathcal T}}

\newcommand{\OK}{\cO_K}
\newcommand{\OD}{\cO_D}
\newcommand{\Opi}{\cO_\pi}
\renewcommand{\a}{\alpha}
\newcommand{\maps}{\colon\thinspace}
\newcommand{\tauC}{\tau_\C}
\newcommand{\tauR}{\tau_\R}
\newcommand{\real}{\mathrm{Re}}
\newcommand{\leftquom}[4]{{\raisebox{-#3}{$#1$}}#4\backslash{\raisebox{#3}{$#2$}}}
\newcommand{\hypquo}[1]{\leftquom{#1}{\H^3}{0.5pt}{\big}}
\DeclareMathOperator{\Isom}{Isom}
\renewcommand{\H}{{\mathbb{H}}}
\renewcommand{\abs}[1]{{\left| #1 \right|}}
\newcommand{\somethingdef}[4]{{\left#1 {#2} \ \left| \ {#3} \vphantom{#2} \right. \right#4}}
\newcommand{\setdef}[2]{\somethingdef{\{}{#1}{#2}{\}}}
\newcommand{\spandef}[2]{\somethingdef{\langle}{#1}{#2}{\rangle}}
\newcommand{\mtext}[1]{\quad \mbox{#1} \quad}

\newcommand{\kernoverline}[3]{{\mkern #1mu\overline{\mkern-#1mu #2\mkern-#3mu}\mkern#3mu}}
\newcommand{\abar}{{\kernoverline{1.5}{a}{0}}}
\newcommand{\bbar}{{\kernoverline{1.0}{b}{-1}}}

\renewcommand{\rel}[1]{\mathit{#1}}
\newcommand{\defgen}[2]{#1 = \mathit{#2}}
\newlength{\pflinesep}
\newcommand{\pfline}[2]{\hspace{-1.0cm}\hspace{#1\pflinesep}#2\par}
\newcommand{\pfassume}[2]{\pfline{#1}{Case $#2 \in P$:}}
\newcommand{\pfassumefinal}[2]{\pfline{#1}{Case $#2 \in P$:  Then $P$ contains the following, which is 1 in $\Gamma$:}}
\newcommand{\pfcontradiction}[2]{\pfline{#1}{$#2$}}

\author[S. Kionke, J. Raimbault]{Steffen Kionke, Jean Raimbault \\ With an appendix by Nathan Dunfield}

  \address{Heinrich-Heine-Universit\"at D\"usseldorf \\ Mathematisches Institut \\ Universit\"atsstr.~1 \\ 40225 D\"usseldorf \\ Germany}
  \email{steffen.kionke@uni-duesseldorf.de} 
  \urladdr{\url{http://www-public.rz.uni-duesseldorf.de/~kionke/}}

  \address{Institut de Math\'ematiques de Toulouse ; UMR5219 \\ Universit\'e de Toulouse ; CNRS \\ UPS IMT, F-31062 Toulouse Cedex 9, France}
  \email{Jean.Raimbault@math.univ-toulouse.fr}
  \urladdr{\url{http://www.math.univ-toulouse.fr/~jraimbau/}}

  \address{ Department of Mathematics, MC-382 \\
          University of Illinois \\
          1409 West Green Street \\
          Urbana, IL 61801 \\ 
          United States}
  \email{nathan@dunfield.info}
  \urladdr{\url{http://dunfield.info}}

 
\title{On geometric aspects of diffuse groups}

\begin{document}

\begin{abstract}
Bowditch introduced the notion of diffuse groups as a geometric variation of the unique product property.
We elaborate on various examples and non-examples, keeping the geometric point of view from Bowditch's paper.
In particular, we discuss fundamental groups of flat and hyperbolic manifolds.
The appendix settles an open question by providing an example of a group which is diffuse but not left-orderable.
\end{abstract}

\maketitle
\setcounter{tocdepth}{1}
\tableofcontents

\setcounter{tocdepth}{2}

%%%%%%%%%%%%%%%%%%%%%%%%%%%%%%%%%%%%%%%%%%%%%%%%%%%%%%%%%%%%%%%%%%%%%%%%%%%%%%%%

\section{Introduction}

Following B. Bowditch \cite{Bowditch_diff} we say that
a group $\Gamma$ is \emph{diffuse}, if every finite non-empty subset $A \subset \Gamma$ has an \emph{extremal point}, this is, an element $a\in A$ 
such that for any $g \in \Gamma\setminus\{1\}$ either $ga$ or $g^{-1}a$ is not in $A$ (see also \ref{surv_Bowditch} below). 
A non-empty finite set without extremal points will be called a \emph{ravel}\footnote{We think of this as an entangled ball of string.};
thus a group is diffuse if and only if it does not contain a ravel. Every non-trivial finite subgroup of $\Gamma$ is a ravel, hence a diffuse group is torsion-free. 
In this work we use geometric methods to  discuss various examples of diffuse and non-diffuse groups.

The interest in diffuse groups stems from Bowditch's observation that they have the unique product property (see Section \ref{sec_relatedProp} below).
Originally unique products were introduced in the study of group rings of discrete, torsion-free groups.
More precisely, it is easily seen that if a group $\Gamma$ has unique products, then
it satisfies Kaplansky's unit conjecture. In simple terms this means that the units in the group ring $\CC[\Gamma]$ are all \emph{trivial},
i.e.~of the form $\lambda g$ with $\lambda \in \CC^\times$ and $g \in \Gamma$. A similar question can be asked replacing $\CC$ by some integral domain.
A weaker conjecture (Kaplansky's zero divisor conjecture) asserts that $\CC[\Gamma]$ contains no zero divisor, and a still weaker one that it contains no idempotents other than $1_\Gamma$.
There are other approaches to the zero divisor and idempotent conjecture (see for example \cite{BLR}, \cite[Chapter 10]{Lueck_book})
which have succeeded in proving it for large classes of groups, whereas the unit conjecture has (to the best of our knowledge) only been tackled by establishing 
the possibly stronger
unique product property for a given group. In consequence 
it is still unknown if the conclusion of the unit conjecture holds, for example, for all torsion-free groups in
the class of crystallographic groups (see \cite{Craven_Pappas} for more on the subject),
while that of the zero-divisor conjecture is known to hold (among other) for all torsion-free groups in the finite-by-solvable class, as proven by P. Kropholler, P. Linnell and J. Moody in \cite{KLM}. 

There are further applications of the unique product property.
For instance, if $\Gamma$ has unique products, then it satisfies a conjecture of Y. O. Hamidoune on the size of isoperimetric atoms (cf.\ Conjecture 10 in \cite{BPS-isoperimetric}).
Let us further mention that there are torsion-free groups without unique products and
there has been some interest generated by finding such groups: see for instance \cite{Rips-Segev},\cite{Promislow},\cite{Steenbock},\cite{Arzhantseva_Steenbock},\cite{Carter}. 
We note that for the examples in \cite{Rips-Segev} (and their generalization in \cite{Steenbock}) it is not known if the conclusion of the zero-divisor conjecture holds. 

Using Lazard's theory of analytic pro-$p$ groups one can show that every arithmetic group $\Gamma$ has a finite index subgroup $\Gamma'$ such that the group ring $\ZZ[\Gamma']$
satisfies the zero divisor conjecture. This work originated from the idea to similarly study Kaplansky's unit conjecture \emph{virtually}.
In this spirit we establish virtual diffuseness for different classes of groups and, moreover, we discuss examples of diffuse and non-diffuse groups in order to clarify the borderline between the two.
Our results are based on geometric considerations.

%%%%%%%%%%%%%%%%%%%%%%%%%%%%%%%%%%%%%%%%%%%%%%%%%%%%%%%%%%%%

\subsection{Results}

\subsubsection{Crystallographic groups}
The torsion-free crystallographic groups, also called Bieberbach groups, are virtually diffuse since free abelian groups are diffuse.
However, already in dimension three there is a Bieberbach group $\Delta_P$ which is not diffuse \cite{Bowditch_diff}. In fact, Promislow even showed that the group $\Delta_P$ 
does not satisfy the unique product property \cite{Promislow}.
On the other hand, the other nine 3-dimensional Bieberbach groups are diffuse. So is there an easy way to decide from a given Bieberbach group
whether it is diffuse or not? In Section \ref{sec_infrasol} we discuss this question and show that in many cases it suffices to know the holonomy group.
\begin{theostar}
 Let $\Gamma$ be a Bieberbach group with holonomy group $G$.
 \begin{itemize}
  \item[(i)]  If $G$ is not solvable, then $\Gamma$ is not diffuse.
  \item[(ii)] If $G$ has only cyclic Sylow subgroups, then $\Gamma$ is diffuse.
 \end{itemize}
\end{theostar}
Note that a finite group $G$ with cyclic Sylow subgroups is meta-cyclic, thus solvable. 
We further show that in the remaining case, where $G$ is solvable and has a non-cyclic Sylow subgroup,
the group $G$ is indeed the holonomy of a diffuse and a non-diffuse Bieberbach group.
Moreover, we give a complete list of the $16$ non-diffuse Bieberbach groups in dimension four.
Our approach is based on the equivalence of diffuseness and local indicability for amenable groups 
as obtained by Linnell and Witte Morris \cite{lW-M}. We include a new geometric proof of their result for the special
case of virtually abelian groups.

%%%%%%%%%%%%%%%%%%%%%%%%%%%%%%

\subsubsection{Discrete subgroups of rank-one Lie groups}

The class of hyperbolic groups is one of the main sources of examples of diffuse groups in \cite{Bowditch_diff}: it is an immediate consequence of Corollary 5.2 
loc.~cit.\ that any residually finite, word-hyperbolic group contains with finite index a diffuse subgroup
(the same statement for unique products was proven earlier by T.~Delzant \cite{Delzant}). In particular, cocompact discrete subgroups of rank one Lie groups are virtually diffuse (for example, given an arithmetic lattice $\Gamma$ in such a Lie group, any normal congruence subgroup of $\Gamma$ of sufficiently high level is diffuse). On the other hand not much is known in this respect about relatively hyperbolic groups, and it is natural to ask whether a group which is hyperbolic relatively to diffuse subgroups must itself be virtually diffuse. In this paper we answer this question by the affirmative in the very particular case of non-uniform lattices of rank one Lie groups. 

\begin{theostar}
If $\Gamma$ is a lattice in one of the Lie groups $\SO(n,1),\SU(n,1)$ or $\Sp(n,1)$ then there is a finite-index subgroup $\Gamma'\le\Gamma$ such that $\Gamma'$ is diffuse. 
\label{rank1_lattice:intro}
\end{theostar}

The proof actually shows, in the case of an arithmetic lattice, that normal congruence subgroups of sufficiently large level are diffuse.
We left the case of non-uniform lattices in the exceptional rank one group $F_4^{-20}$ open, but it is almost certain that the proof of this
Theorem adapts also to this case. Theorem \ref{rank1_lattice:intro} is obtained as a corollary of a result on a more general class
of geometrically finite groups of isometries. Another consequence is the following theorem.

\begin{theostar}
Let $\Gamma$ be any discrete, finitely generated subgroup of $\SL_2(\CC)$. There exists a finite-index subgroup $\Gamma'\le\Gamma$ such that $\Gamma'$ is diffuse. 
\label{geomfin_dim3:intro}
\end{theostar}

The proof of these theorems uses the same approach as Bowditch's, 
that is a metric criterion (Lemma \ref{dist_move_crit} below) for the action on the relevant hyperbolic space. 
The main new point we have to establish concerns the behaviour of unipotent isometries: 
the result we need (Proposition \ref{parsep} below) is fairly easy to observe for real hyperbolic spaces;
for complex ones it follows from a theorem of M. Phillips \cite{Phillips_dir}, and we show that the argument used there can be generalized in a straightforward way to quaternionic hyperbolic spaces. We also study axial isometries of real hyperbolic spaces in some detail, and give an optimal criterion (Proposition \ref{bowditch_mieux}) which may be of use in determining whether a given hyperbolic manifold has a diffuse fundamental groups. 

%%%%%%%%%%%%%%%%%%%%%%%%%%%%%%

\subsubsection{Three--manifold groups}

Following the solution of both Thurston's Geometrization
conjecture (by G. Perelman \cite{Perelman1, Perelman2})
and the Virtually Haken conjecture (by I. Agol \cite{Agol_VH} building on work of D. Wise) it is known by previous work
of J. Howie \cite{Howie}, 
and S. Boyer, D. Rolfsen and B. Wiest \cite{BRW}
that the fundamental group of any compact three--manifold contains
a left-orderable finite-index subgroup.
Since left-orderable groups are diffuse (see Section \ref{sec_relatedProp} below)
this implies the following.

\begin{theostar}
\label{thm_3manifolds}
Let $M$ be a compact three--manifold, then there is a finite-index subgroup in $\pi_1(M)$ which is diffuse. 
\label{res_dim3}
\end{theostar}

Actually, one does not need Agol's work to prove this weaker result:
the case of irreducible manifolds with non-trivial JSJ-decomposition
is dealt with in \cite[Theorem 1.1(2)]{BRW},
and non-hyperbolic geometric manifolds are easily
seen to be virtually orderable.
Finally, closed hyperbolic manifolds can be handled
by Bowditch's result (see (iv) in Section \ref{surv_Bowditch} below).

We give a more direct proof of Theorem \ref{res_dim3} in Section \ref{3_mfd}; 
the tools we use (mainly a `virtual' gluing lemma) may be of independent interest.
The relation between diffuseness (or unique products) and left-orderability is not very clear at present;
in the appendix Nathan Dunfield gives an example of a compact hyperbolic three-manifold whose fundamental group is not left-orderable, but nonetheless diffuse. 

\subsection*{Acknowledgements}

We are pleased to thank to George Bergman, Andres Navas and Markus Steenbock for valuable comments on a first version of this paper. The second author would especially like to thank Pierre Will for directing him to the article \cite{Phillips_dir}. 

Both authors are grateful to the Max-Planck-Institut f\"ur Mathematik in Bonn, where this work was initially developed, and which supported them financially during this phase.

%%%%%%%%%%%%%%%%%%%%%%%%%%%%%%%%%%%%%%%%%%%%%%%%%%%%%%%%%%%%%%%%%%%%%%%%%%%%%%%%

\section{Diffuse groups}
\label{diffuse-general}

We review briefly various notions and works related to diffuseness and present some questions and related examples of groups. 

%%%%%%%%%%%%%%%%%%%%%%%%%%%%%%%%%%%%%%%%%%%%%%%%%%%%%%%%%%%%

\subsection{A quick survey of Bowditch's paper}
\label{surv_Bowditch}

We give here a short recapitulation of some of the content in Bowditch's paper~\cite{Bowditch_diff}.
It introduces the general notion of a diffuse action of a group. Let $\Gamma$ be a group acting on a set $X$.
Given a finite set $A \subset X$. An element $a \in A$ is said to be an \emph{extremal point} in $A$, 
if for all $g \in \Gamma$ which do not stabilize $a$ either $ga$ or $g^{-1}a$ is not in $A$.
The action of $\Gamma$ on $X$ is said to be diffuse
if every finite subset $A$ of $X$ with $|A| \geq 2$ has at least {\it two} extremal points. An action in which each finite subset has at least {\it one} extremal point is called weakly diffuse by Bowditch; we will not use this notion in the sequel. 
It was observed by Linnell and Witte-Morris \cite[Prop.6.2.]{lW-M}, that a free action is diffuse if and only if it is weakly diffuse.
Thus a group is diffuse (in the sense given in the introduction) if and only if its action on itself by left-translations is diffuse.
More generally, Bowditch proves that if a group admits a diffuse action whose stabilizers are diffuse groups, then the group itself is diffuse.
In particular, an extension of diffuse groups is diffuse as well. 

The above can be used to deduce the diffuseness of many groups. For example,
 strongly polycyclic groups are diffuse since they are, by definition, obtained from the trivial group by taking successive extensions by $\ZZ$. 
Bowditch's paper also provides many more examples of diffuse groups:
\begin{enumerate}[(i)]
\item The fundamental group of a compact surface of nonpositive Euler characteristic is diffuse; 
\item More generally, any free isometric action of a group on an $\RR$-tree is diffuse; 
\item A free product of two diffuse groups is itself diffuse; 
\item A closed hyperbolic manifold with injectivity radius larger than $\log(1+\sqrt{2})$ has a diffuse fundamental group.  
\end{enumerate}

We conclude this section with the following simple useful lemma, which appears as Lemma 5.1 in \cite{Bowditch_diff}.

\begin{lemma}
If $\Gamma$ acts on a metric space $(X,d_X)$ satisfying the condition
\begin{equation}
\forall x,y\in X,g\in\Gamma :\: gx \neq x \implies \max(d_X(gx,y), d_X(g^{-1}x,y)) > d(x,y)
\label{dist_move_crit_eq}\tag{$\ast$}
\end{equation}
then the action is diffuse. 
\label{dist_move_crit}
\end{lemma} 
\begin{proof}
 Let $A \subset X$ be compact with at least two elements. Take $a,b$ in $A$ with $d(a,b) = \rm{diam}(A)$, then these are extremal in $A$.
 It suffices to check this for $a$. Given $g \in \Gamma$ not stabilizing $a$, then $ga$ or $g^{-1}a$ is farther away from $b$, hence not in~$A$. 
\end{proof}

 Note that this argument does not require nor that the action be isometric, neither that the function $d_X$ on $X\times X$ be a distance.
 However this geometric statement is sufficient for all our concerns in this paper.

%%%%%%%%%%%%%%%%%%%%%%%%%%%%%%%%%%%%%%%%%%%%%%%%%%%%%%%%%%%%

\subsection{Related properties}\label{sec_relatedProp}

Various properties of groups have been defined, which are closely related to diffuseness.
We remind the reader of some of these properties and their mutual relations.

Let $\Gamma$ be a group. 
We say that $\Gamma$ is \emph{locally indicable}, if every finitely generated non-trivial subgroup admits a non-trivial homomorphism into the group $\ZZ$. 
In other words, every finitely generated subgroup of $\Gamma$ has a positive first rational Betti number.

Let $\prec$ be a total order on $\Gamma$. The order is called \emph{left invariant}, if
\begin{equation*}
    x \prec y \:\implies \: gx \prec gy
\end{equation*}
for all $x$, $y$ and $g$ in $\Gamma$.
We say that the order $\prec$ on $\Gamma$ is \emph{locally invariant} if for all $x, g \in \Gamma$ with $g \neq 1$
either $gx \prec x$ or $g^{-1}x \prec x$.
Not all torsion-free groups admit orders with one of these properties. We say that $\Gamma$ is \emph{left-orderable} (resp.\ LIO) 
if there exists a left-invariant (resp.\ locally invariant) order on $\Gamma$.
It is easily seen that an LIO group is diffuse.
In fact, it was pointed out by Linnell and Witte Morris \cite{lW-M} that a group is LIO if and only if it is diffuse.
One can see this as follows: If $\Gamma$ is diffuse then every finite subset admits a locally invariant order (in an appropriate sense), and this 
yields a locally invariant order on $\Gamma$ by a compactness argument.

The group $\Gamma$ is said to have the \emph{unique product property} (or to have unique products) if for every two finite non-empty subsets $A, B \subset \Gamma$
there is an element in the product $x \in A\cdot B$ which can be written uniquely in the form $x = a b$ with $a \in A$ and $b \in B$.

The following implications are well-known:
\begin{equation*}
 \text{locally indicable} \stackrel{(1)}{\implies} \text{left-orderable} \stackrel{(2)}{\implies} \text{diffuse} \stackrel{(3)}{\implies}\text{unique products}
\end{equation*}
An example of Bergman \cite{Bergman} shows that (1) is in general not an equivalence, i.e.\ there are left-orderable groups which are not locally indicable (further examples are given by some of the hyperbolic three--manifolds studied in \cite[Section 10]{CaDu} which have a left-orderable fundamental group with finite abelianization). 

An explicit example showing that (2) is not an equivalence either is explained 
in the appendix written by Nathan Dunfield (see Theorem \ref{thm:hypmain}). 
However, the reverse implication to (3), that is the relation between unique products and diffuseness, remains completely mysterious to us. We have no idea what the answer to the following question should be
(even by restricting to groups in a smaller class, for example crystallographic, amenable, linear or hyperbolic groups). 

\begin{question}
Does there exist a group which is not diffuse but has unique products? 
\end{question}
 
It seems extremely hard to verify, for a given group, the unique product property without using any of the other three properties.

%%%%%%%%%%%%%%%%%%%%%%%%%%%%%%%%%%%%%%%%%%%%%%%%%%%%%%%%%%%%
\newpage
\subsection{Some particular hyperbolic three--manifolds}
\subsubsection{A diffuse, non-orderable group}

In the appendix Nathan Dunfield describes explicitly an example of an arithmetic Kleinian group which is diffuse but not left-orderable -- this yields the following result (Theorem \ref{thm:hypmain}).

\begin{theorem}[Dunfield]
There exists a finitely presented (hyperbolic) group which is diffuse but not left-orderable. 
\label{diff_not_ord}
\end{theorem}

With Linnell and Witte-Morris' result that for amenable groups both properties are equivalent \cite{lW-M} this shows that there is a difference in these matters between amenable and hyperbolic groups. To verify that the group is diffuse one can use Bowditch's result or our Proposition \ref{bowditch_mieux}.

Let us make a few comments on the origins of this example. The possibility to find such a group among this class of examples was proposed, unbeknownst to the authors, by A. Navas---see \cite[1.4.3]{Deroin_Navas_Rivas}. 
Nathan Dunfield had previously computed a vast list of examples of closed hyperbolic three--manifolds whose fundamental group is not left-orderable (for some examples see \cite{CaDu}), using an algorithm described in the second paper. The example in Appendix \ref{appendix} was not in this list, but was obtained by searching through the towers of finite covers of hyperbolic \3-manifolds studied in \cite[\S 6]{CalegariDunfield2006}.

%More precisely, the manifold that Nathan Dunfield indicated to us is the following:let $M$ be the manifold obtained as the $(\ZZ/3)^2$-cover of the manifold known in the Hodgson-Weeks census\begin{com}Is the Hodgson-Weeks census really something one can refer to? Is there a source to cite with a description of m007(3,2)? A?: I think we should out the technical details in an appendix ; manifolds in the HW census are available in Snappea and we should redo Dunfield's computation and put it along the othr computational stuff \end{com} as m007(3,2)(the volume of $M$ is about 14.248). Then $M$ has systole larger than $1.80 > 2\log(1+\sqrt 2))$ and hence $\pi_1(M)$ is diffuse by Bowditch's result (vi) above (or the case $r=-1$ of Proposition \ref{bowditch_mieux} below). 

\subsubsection{A non-diffuse lattice in $\PSL_2(\CC)$}
We also found an example of a compact hyperbolic 3-manifold with a non-diffuse fundamental group; in fact it is the hyperbolic three--manifold of smallest volume. 

\begin{theorem}
The fundamental group of the Weeks manifold is not diffuse. 
\label{Weeks_nd}
\end{theorem}

We verified this result by explicitly computing a ravel in the fundamental group of the Weeks manifold.
We describe the algorithm and its implementations in Section \ref{computation_ravel}.
In fact, given a group $\Gamma$ and a finite subset $A$ one can decide whether $A$ contains a ravel by the following procedure:
choose a random point $a\in A$ ;
if it is extremal (which we check using a sub-algorithm based on the solution to the word problem in $\Gamma$)
we iterate the algorithm on $A\setminus\{a\}$,
otherwise we continue with another one. Once all the points of $A$ have been tested, what remains is either empty or a ravel in $\Gamma$.

\subsubsection{Arithmetic Kleinian groups}
In a follow-up to this paper we will investigate the diffuseness properties of arithmetic Kleinian groups, in the hope of finding more examples of the above phenomena. Let us mention two results that will be proven there:
\begin{itemize}
\item[(i)] Let $p>2$ be a prime. There is a constant $C_p$ such that if $\Gamma$ is a torsion-free arithmetic group with invariant trace field $F$ of degree $p$ and discriminant $D_F>C_p$, then $\Gamma$ is diffuse.
\item[(ii)] If $\Gamma$ is a torsion-free Kleinian group derived from a quaternion algebra over an imaginary quadratic field $F$ such that
$$
D_F\neq -3,-4,-7,-8, -11, -15, -20, -24
$$
then $\Gamma$ is diffuse. 
\end{itemize}

%%%%%%%%%%%%%%%%%%%%%%%%%%%%%%%%%%%%%%%%%%%%%%%%%%%%%%%%%%%%

\subsection{Groups which are not virtually diffuse}

All groups considered in this article are residually finite and turn out to be virtually diffuse.
Due to a lack of examples, we are curios about an answer to the following question.

\begin{question}
  Is there a finitely generated (resp.\ finitely presented) group which is torsion-free, residually finite and not virtually diffuse?
\end{question}

The answer is positive without the finiteness hypotheses:
 given any non-diffuse, torsion-free, residually finite group $\Gamma$, then an infinite restricted direct product of factors isomorphic to $\Gamma$ is residually finite and not virtually diffuse. 

Furthermore, if we do not require the group to be residually finite, then
one may take a
restricted wreath product $\Gamma \wr U$ with some infinite group $U$.
The group $\Gamma \wr U$ is not virtually diffuse and it is finitely generated if $\Gamma$ and $U$ are finitely generated (not finitely presented, however). 
Moreover, by a theorem of Gruenberg \cite{Gruenberg} such a wreath product ($\Gamma$ non-abelian, $U$ infinite) is not residually finite.
Other examples of groups which are not virtually diffuse are the amenable simple groups constructed by K. Juschenko and N. Monod in \cite{Juschenko_Monod}; 
these groups cannot be locally indicable, however they are neither residually finite nor finitely presented. 

In the case of hyperbolic groups, this question is related to the residual properties of these groups -- namely it is still not known if all hyperbolic groups are residually finite.
A hyperbolic group which is not virtually diffuse
would thus be, in light of the results of Delzant--Bowditch, not residually finite.
It is unclear to the authors if this approach is feasible; for results in this direction see \cite{Gruber_Martin_Steenbock}. 

Finally, let us note that it would also be interesting to study the more restrictive class of linear groups instead of residually finite ones.

%%%%%%%%%%%%%%%%%%%%%%%%%%%%%%%%%%%%%%%%%%%%%%%%%%%%%%%%%%%%

%%%%%%%%%%%%%%%%%%%%%%%%%%%%%%%%%%%%%%%%%%%%%%%%%%%%%%%%%%%%%%%%%%%%%%%%%%%%%%%%

\section{Fundamental groups of infra-solvmanifolds}
\label{sec_infrasol}

%%%%%%%%%%%%%%%% Solvable groups and infra solvmanifolds
\subsection{Introduction}
\subsubsection{Infra-solvmanifolds}
In this section we discuss diffuse and non-diffuse fundamental groups of infra-solvmanifolds. 
The focus lies on crystallographic groups, however we shall begin the discussion in a more general setting.
Let $G$ be a connected, simply connected, solvable Lie group and let $\aut(G)$ denote the group of continuous automorphisms of $G$. The affine group of $G$ is the semidirect product
$\aff(G) = G \rtimes \aut(G)$. A \emph{lattice} $\Gamma \subset G$ is a discrete cocompact subgroup of $G$.
An \emph{infra-solvmanifold} (of type $G$) is a quotient manifold $G/\Lambda$ where $\Lambda \subseteq \aff(G)$ is a torsion-free subgroup of the affine group 
such that $\Lambda \cap G$ has finite index in $\Lambda$ and is a lattice in $G$. If $\Lambda$ is not diffuse, we say that $G/\Lambda$ is a non-diffuse
infra-solvmanifold. 

The compact infra-solvmanifolds which come from a nilpotent Lie group
$G$ are characterised by the property that they are {\it almost flat}: that is,
they admit Riemannian metrics with bounded diameter and arbitrarily small sectional curvatures 
(this is a theorem of M. Gromov, see \cite{Gromov_flat}, \cite{Buser_Karcher}).
Those that come from abelian $G$ are exactly those that are {\it flat}, i.e. they admit a Riemannian metric with vanishing sectional curvatures.
We will study the latter in detail further in this section. We are not aware of any geometric characterization of general infra-solvmanifolds. 

%%%%%%%%%%%%%%%%%%%%%%%%%%%%%%

\subsubsection{Diffuse virtually polycyclic groups are strongly polycyclic}
Recall that a group $\Gamma$ is (strongly) polycyclic if it admits a subnormal series with (infinite) cyclic factors.
By a result of Mostow lattices in connected solvable Lie groups are polycyclic (cf.\ Prop.~3.7 in \cite{Raghunathan1972}). 
Consequently, the fundamental group of an infra-solvmanifold 
is a virtually polycyclic group. 

As virtually polycyclic groups are amenable, we can use the following striking result 
of Linnell and Witte Morris \cite{lW-M}. 
\begin{theorem}[Linnell, Witte Morris]\label{thm_LinnellWitteMorris}
  An amenable group is diffuse if and only if it is locally indicable.
\end{theorem}
We shall give a geometric proof of this theorem for the special case of virtually abelian groups in the next section.
Here we confine ourselves to pointing out the following algebraic consequence.
\begin{prop}\label{prop_diffusePolycyclic}
 A virtually polycyclic group $\Gamma$
 is diffuse if and only if $\Gamma$ is strongly polycyclic.
 Consequently, the fundamental group of an infra-solvmanifold is diffuse exactly if it is strongly polycyclic.
\end{prop}
\begin{proof}
 Clearly, a strongly polycyclic group is a virtually polycyclic group, in addition
 it is diffuse by Theorem 1.2 in \cite{Bowditch_diff}.

Assume that $\Gamma$ is diffuse and virtually polycyclic. 
We show that $\Gamma$ is strongly polycyclic by induction on the Hirsch length $h(\Gamma)$.
If $h(\Gamma) = 0$, then $\Gamma$ is a finite group and as such it can only be diffuse if it is trivial.

Suppose $h(\Gamma)=n >0$ and suppose that the claim holds for all groups of Hirsch length at most $n-1$. 
By Theorem \ref{thm_LinnellWitteMorris} $\Gamma$ is locally indicable and (since $\Gamma$ is finitely generated)
we can find a surjective homomorphism $\phi\colon \Gamma \to \ZZ$.
Observe that $h(\Gamma) = h(\ker(\phi)) + 1$.
The kernel $\ker(\phi)$ is diffuse and virtually polycyclic, and we deduce from the induction hypothesis,
that $\ker(\phi)$ (and so $\Gamma$) is strongly polycyclic.
\end{proof}

In the next three sections we focus on crystallographic groups. After the discussion of a geometric proof of Theorem
\ref{thm_LinnellWitteMorris} in the crystallographic setting (\ref{sec_geometric_ravels}),
we will analyse the influence of the structure of the holonomy group for the existence of ravels (\ref{sec_Holonomy}).
We also give a list of all non-diffuse crystallographic groups in dimension up to four (\ref{sec_lowDimensions}).
Finally, we discuss a family of non-diffuse infra-solvmanifolds in \ref{sec_ExamplesOfSol} which are not flat manifolds.

\subsection{Geometric construction of ravels in virtually abelian groups}\label{sec_geometric_ravels}
   
   The equivalence of local indicability and diffuseness for amenable groups which was established by Linnell and
   Witte Morris \cite{lW-M} is a powerful result.
   Accordingly a virtually polycyclic group with vanishing first rational Betti number contains a ravel. 
   However, their proof does not explain a construction of ravels based on the vanishing Betti number. 
   They stress that this does not seem to be obvious even for virtually abelian groups. 
   The purpose of this section is to give a geometric and elementary proof of this
   theorem, for the special case of virtually abelian groups, which is based on an explicit construction of ravels.
   
   \begin{theorem} \label{thm_LWM_abelian}
    A virtually abelian group is diffuse exactly if it is locally indicable.
   \end{theorem}
   It is well known that local indicability implies diffuseness. It suffices to prove the converse.
   Let $\Gamma_0$ be a virtually abelian group and assume that it is not locally indicable.
   We can find a finitely generated subgroup $\Gamma \subset \Gamma_0$ with vanishing first rational Betti number.
   If $\Gamma$ contains torsion, it is not diffuse. Thus we assume that $\Gamma$ is torsion-free.
   Since a finitely generated torsion-free virtually abelian group is crystallographic,
   the theorem follows from the next lemma.
   
   \begin{lemma}\label{ravel_crystal}
     Let $\Gamma$ be a crystallographic group acting on a euclidean space $E$.
     If $b_1(\Gamma)=0$, then for all $e \in E$ and all sufficiently large $r >0 $ the set 
      \begin{equation*}
           B(r,e) = \{\:\gamma \in \Gamma\:|\: \|\gamma e - e\| \leq r  \:\}
      \end{equation*}
      is a ravel.
   \end{lemma}

   \begin{proof}
    We can assume $e=0 \in E$.
    Let $\Gamma$ be a non-trivial crystallographic group with vanishing first Betti number and let $\pi: \Gamma \to G$ be the projection
    onto the point group at $0$. The translation subgroup is denoted by $T$ and we fix some $r_0 >0$ so that for every $u \in E$
    there is $t\in T$ satisfying $\|u-t\| \leq r_0$.

    The first Betti number $b_1(\Gamma)$ is exactly the dimension of the space $E^G$ of $G$-fixed vectors.
    Thus $b_1(\Gamma)=0$ means that $G$ acts without non-trivial fixed points on $E$.
    Since every non-zero vector is moved by $G$, there is a real number $\delta < 1$ such that 
    for all $u \in E$ there is $g\in G$ such that 
    \begin{equation}\label{eq_approximation}
       \|gu+u\| \leq 2 \delta \|u\|. 
    \end{equation}
    For $r>0$ let $B_r$ denote the closed ball of radius $r$ around $0$. 
    Fix $u \in B_r$; we shall find $\gamma \in \Gamma$ such that $\|\gamma u\| \leq r$ and $\|\gamma^{-1}u\| \leq r$ provided $r$ is sufficiently large.
    We pick $g\in G$ as in \eqref{eq_approximation} and we choose some
    $\gamma_0 \in \Gamma$ with $\pi(\gamma_0) = g$. Define $w_0 = \gamma_0(0)$.
    
    We observe that for every two vectors $v_1, v_2 \in E$ with distance $d$, there is
    $x \in w_0 + T$ with
    \begin{equation*}
         \max_{i=1,2}( \|v_i-x\| ) \leq r_0 + \frac{d}{2}.
    \end{equation*}
    Indeed, the ball of radius $r_0$ around the midpoint of the line between $v_1$ and $v_2$ contains an element $x \in w_0 + T$.
    Apply this to the vectors $v_1 = u$ and $v_2 = -gu$ to find some $x = w_0 + t$. By construction we get $d \leq 2\delta r$.
    
    Finally we define $\gamma = t\cdot \gamma_0$ to deduce the inequalities
    \begin{equation*}
      \|\gamma u\| = \|gu + x\| = \|-gu - x\| \leq r_0 + \delta r
    \end{equation*}
    and
    \begin{equation*}
      \|\gamma^{-1} u\| = \|g^{-1}u - g^{-1}x\| = \|u - x\| \leq r_0 + \delta r.
    \end{equation*}
    As $\delta < 1$ the right hand side is less than $r$ for all sufficiently large~$r$.
   \end{proof}

\subsection{Diffuseness and the holonomy of crystallographic groups}\label{sec_Holonomy}

We take a closer look at the non-diffuse crystallographic groups and their holonomy groups.
It will turn out that for a given crystallographic group one can often decide from the holonomy group whether or not the group is diffuse.
In the following a \emph{Bieberbach group} is a non-trivial torsion-free crystallographic group.
Let $\Gamma$ be a Bieberbach group, it has a finite index normal maximal abelian subgroup $T \subset \Gamma$. Recall that the finite quotient
$G = \Gamma/T$ is called the holonomy group of $\Gamma$.
Since every finite group is the holonomy group of some Bieberbach group (by a result due to Auslander-Kuranishi \cite{AuslanderKuranishi}), 
this naturally divides the finite groups into three classes.
\begin{definition}
  A finite group $G$ is \emph{holonomy diffuse} if every Bieberbach group $\Gamma$ with holonomy group $G$ is diffuse.
  It is \emph{holonomy anti-diffuse} if every Bieberbach group $\Gamma$ with holonomy group $G$ is non-diffuse.
  Otherwise we say that $G$ is \emph{holonomy mixed}.
\end{definition}
For example, the finite group $(\ZZ/2\ZZ)^2$ is holonomy mixed.
In fact, the Promislow group $\Delta_P$ (also known as Hantzche-Wendt group or Passman group) is a non-diffuse \cite{Bowditch_diff} Bieberbach group with holonomy group $(\ZZ/2\ZZ)^2$ -- thus $(\ZZ/2\ZZ)^2$ is not holonomy diffuse.
On the other hand it is easy to construct diffuse groups with holonomy group $(\ZZ/2\ZZ)^2$ (cf.\ Lemma~\ref{lem_solvable-holonomy} below). 

In this section we prove the following algebraic characteristion of these three classes of finite groups.
\begin{theorem}\label{thm_holonomy}
 A finite group $G$ is
 \begin{enumerate}[(i)]
  \item \label{holantidiff}  holonomy anti-diffuse if and only if it is not solvable.
  \item \label{holdiff}  holonomy diffuse exactly if every Sylow subgroup is cyclic.
  \item \label{holmixed}  holonomy mixed if and only if it is solvable and has a non-cyclic Sylow subgroup.
 \end{enumerate}
\end{theorem}

\begin{remark}
  A finite group $G$ with cyclic Sylow subgroups is meta-cyclic (Thm.~9.4.3 in \cite{Hall_Groups}). 
  In particular, such a group $G$ is solvable.
\end{remark}

The proof of this theorem will be given as a sequence of lemmata below.
It suffices to prove the assertions \eqref{holantidiff} and \eqref{holdiff}.
One direction of \eqref{holantidiff} is easy. By Proposition \ref{prop_diffusePolycyclic}
a diffuse Bieberbach group is solvable and thus cannot have a finite non-solvable quotient, i.e.~a non-solvable group is holonomy anti-diffuse.
For \eqref{holantidiff} it remains to verify that every finite solvable group is the holonomy of some diffuse Bieberbach group; this will be done in Lemma \ref{lem_solvable-holonomy}.

In order to prove \eqref{holdiff},
we shall use a terminology introduced by Hiller-Sah~\cite{HillerSah1986}.
\begin{definition}
 A finite group $G$ is \emph{primitive} if it is the holonomy group of a Bieberbach group with finite abelianization.
\end{definition}
Statement \eqref{holdiff} of the theorem will follow from the next lemma.
\begin{lemma}\label{lem_holdiff}
Let $G$ be a finite group. The following statements are equivalent.
\begin{enumerate}[  ~~(a)]
 \item $G$ is not holonomy diffuse.\label{lem_item_holdiff}
 \item $G$ has a non-cyclic Sylow subgroup.\label{lem_item_sylow}
 \item $G$ contains a normal primitive subgroup.\label{lem_item_primitve}
\end{enumerate}
\end{lemma}

We frequently use the following notion: A cohomology class $\alpha \in H^2(G,A)$ (for some finite group $G$ and some $G$-module $A$) is called \emph{special}
if it corresponds to a torsion-free extension of $G$ by $A$. Equivalently, if $A$ is free abelian, the restriction of $\alpha$ to any cyclic subgroup of $G$ is non-zero.

\begin{proof}
  Hiller-Sah \cite{HillerSah1986} obtained an algebraic characteristion of primitive groups.
  They showed that a finite group is primitive exactly if it does not contain a cyclic Sylow $p$-subgroup 
  which admits a normal complement (see also \cite{CliffWeiss1989} for a different criterion).
  
  \eqref{lem_item_holdiff} $\implies$ \eqref{lem_item_sylow}: Assume $G$ is not holonomy diffuse and take a non-diffuse Bieberbach group $\Gamma$ with holonomy group $G$.
  As $\Gamma$ is not locally indicable we find a non-trivial subgroup $\Gamma_0 \leq \Gamma$ with $b_1(\Gamma_0) = 0$.
  The holonomy group $G_0$ of $\Gamma_0$ is primitive. Let $p$ be the smallest prime divisor of $|G_0|$. The Sylow $p$-subgroups of $G_0$ are not cyclic, since
  otherwise they would admit a normal complement (by a result of Burnside \cite{Burnside1895}).
  Let $\pi: \Gamma \to G$ be the projection. The image $\pi(\Gamma_0)$ has $G_0$ as a quotient and hence $\pi(\Gamma_0)$ also has non-cyclic Sylow $p$-subgroups.
  As every $p$-group is contained in a Sylow $p$-subgroup, we deduce that the Sylow $p$-subgroups of $G$ are not cyclic.
  
  \medskip
   \eqref{lem_item_sylow} $\implies$ \eqref{lem_item_primitve}: Let $p$ be a prime such that the Sylow $p$-subgroups of $G$ are not cyclic.
   Consider the subgroup $H$ of $G$ generated by all $p$-Sylow subgroups. The group $H$ is normal in $G$ and we claim that it is primitive.
   The Sylow $p$-subgroups of $H$ are precisely those of $G$ and they are not cyclic. Let $p'$ be a prime divisor of $|H|$ different from $p$.
   Suppose there is a (cyclic) Sylow $p'$-subgroup $Q$ in $H$ which admits a normal complement $N$.
   As $H/N$ is a $p'$-group, the Sylow $p$-subgroups of $H$ lie in $N$. By construction $H$ is generated by its Sylow $p$-subgroups and so $N = H$.
   This contradicts the existence of such a Sylow $p'$-subgroup.
   
\medskip
\eqref{lem_item_primitve} $\implies$ \eqref{lem_item_holdiff}: 
Assume now that $G$ contains a normal subgroup $N \normal G$ which is primitive. We show that $G$ is not holonomy diffuse.
Since $N$ is primitive, there exists Bieberbach group $\Lambda$ with point group $N$ and with $b_1(\Lambda,\QQ)=0$.
Let $A$ be the translation subgroup of $\Lambda$ and let $\alpha \in H^2(N,A)$ be the special class corresponding to the extension $\Lambda$.
The vanishing Betti number $b_1(\Lambda,\QQ)=0$ is equivalent to $A^N = \{0\}$.

Consider the induced $\ZZ[G]$-module $B := \ind_N^G(A)$. Let $T$ be a transversal of $N$ in $G$ containing $1_G$. If we restrict the action on $B$ to $N$ we obtain
\begin{equation*}
   B_{|N} = \bigoplus_{g \in T} A(g)
\end{equation*}
where $A(g)$ is the $N$-module obtained from $A$ by twisting with the action with $g$, i.e. $h \in N$ acts by $g^{-1}hg$ on $A$.
In particular, $B^N = \{0\}$ and $A = A(1_G)$ is a direct summand of $B_{|N}$.

Observe that every class in $H^2(N,B)$ which projects to $\alpha \in H^2(N,A)$ is special and defines thus a Bieberbach group with finite abelianization.
Shapiro's isomorphism $\sh^2\colon H^2(G,B) \to H^2(N,A)$ is the composition of the restriction $\res_G^N$ and the projection $H^2(N,B) \to H^2(N,A)$.
We deduce that there is a class $\gamma \in H^2(G,B)$ which maps to some special class $\beta \in H^2(N,B)$ (which projects onto $\alpha \in H^2(N,A)$).
Let $\Lambda'$ be the Bieberbach group (with $b_1(\Lambda') = 0$) corresponding to $\beta$.
The group corresponding to $\gamma$ might not be torsion-free, so we need to vary $\gamma$ so that it becomes a special class.

Let $\frH$ be the collection of all cyclic prime order subgroups $C$ of $G$ which intersect $N$ trivially.
For each $C \in \frH$ we define 
\begin{equation*}
     M_C := \ind_C^G(\ZZ)
\end{equation*}
where $C$ acts trivially on $\ZZ$.
The group $N$ acts freely on $C\backslash G$, since $C\cap N = \{1_G\}$. Therefore $(M_C)_{|N}$ is a free $\ZZ[N]$-module.
We define the $\ZZ[G]$-module
\begin{equation*}
    M = B \oplus \bigoplus_{C \in \frH} M_C.
\end{equation*}
Using Shapiro's Lemma we find classes $\alpha_C \in H^2(G,M_C)$ which restrict to non-trivial classes in $H^2(C,M_C)$.
Consider the cohomology class $\delta := \gamma \oplus \bigoplus_{C \in \frH} \alpha_C \in H^2(G,M)$. 

The class $\delta$ is special, as can be seen as follows. For every $C \in \frH$ this follows from the fact that $\alpha_C$ restricts non-trivially to $C$.
For the cyclic subgroups $C \leq N$ this holds since the restriction of $\gamma$ to $N$ is special.
Consequently $\delta$ defines a Bieberbach group $\Gamma$ with holonomy group $G$.

Finally, we claim that $\res_G^N(\delta) = i_*(\res_G^N(\gamma))$ where $i: B \to M$ is the inclusion map.
Indeed, $H^2(N,M_C) = 0$ since $M_C$ is a free $\ZZ[N]$-module.
Since $\res_G^N(\gamma) = \beta$ we conclude that $\Gamma$ contains the group $\Lambda'$ as a subgroup and thus $\Gamma$ is not locally indicable.
\end{proof}

   We are left with constructing diffuse Bieberbach groups for a given solvable holonomy group.
   We start with a simple lemma concerning fibre products of groups.
   For $0 \leq i \leq n$ let $\Gamma_i$ be a group with a surjective homomorphism $\psi_i$ onto some fixed group $G$.
   The fibre product $\times_G \Gamma_i$ is defined as a subgroup of the direct product $\prod_{i} \Gamma_i$ by
   \begin{equation*}
         \times_G \Gamma_i := \{\:(\gamma_i)_i \in \prod^n_{i=0} \Gamma_i\:|\: \psi_i(\gamma_i) = \psi_0(\gamma_0) \: \text{ for all $i$}\:\}.
   \end{equation*}
    In this setting we observe the following
   \begin{lemma}\label{lem:diffuseFibre}
      If $\Gamma_0$ is diffuse and $\ker{\psi_i} \subset \Gamma_i$ is diffuse for all $i \in\{1,\dots,n\}$, then $\times_G \Gamma_i$ is diffuse.
   \end{lemma}
   \begin{proof}
       There is a short exact sequence
        \begin{equation*}
          1 \longrightarrow \prod^n_{i=1}\ker{\psi_i} \stackrel{j}{\longrightarrow} \times_G \Gamma_i \longrightarrow \Gamma_0 \longrightarrow 1
   \end{equation*}
       so the claim follows from Theorem 1.2 in \cite{Bowditch_diff}.     
    \end{proof}
    \begin{lemma}\label{lem:diffuseclasses}
       Let $G$ be a finite group and let $M_1, \dots, M_n$ be free $\ZZ$-modules with $G$-action.
      Let $\alpha_i \in H^2(G,M_i)$ be classes. If one of these classes defines a diffuse extension group of $G$, then the sum of the $\alpha_i$
      in $H^2(G,M_1\oplus \cdots \oplus M_n)$ defines a diffuse extension of $G$.
   \end{lemma}
   \begin{proof}
       Taking the sum of classes corresponds to the formation of fibre products of the associated extensions, so the claim follows from Lemma \ref{lem:diffuseFibre}.
   \end{proof}

   \begin{lemma} \label{lem_solvable-holonomy}
   Every finite solvable group is the holonomy group of a diffuse Bieberbach group.
   \end{lemma}
   \begin{proof}
   
    We begin by constructing diffuse Bieberbach groups with given abelian holonomy group.
    Let $A$ be an abelian group and let $\Gamma_0$ be a Bieberbach group with holonomy group $A$.
    Write $A$ as a quotient of a free abelian group $\ZZ^k$ of finite rank with projection $\phi\colon \ZZ^k \to A$.
    By Lemma \ref{lem:diffuseFibre} the fibred product $\Gamma_0 \times_A \ZZ^k$ is a diffuse Bieberbach group with holonomy group $A$.
   
    Assume now that $G$ is solvable. We construct a diffuse Bieberbach group $\Gamma$ with holonomy group $G$.
    We will proceed by induction on the derived length of~$G$. The basis for the induction is given by the construction for abelian groups above.
    Let $G'$ be the derived group of $G$. By induction hypothesis there is a faithful $G'$-module $M$ and a ``diffuse'' class
    $\alpha \in H^2(G',M)$. Consider the induced module $B = \ind_{G'}^G(M)$. The restriction of $B$ to $G'$
    decomposes into a direct sum
   \begin{equation*}
        B_{|G'} \cong M \oplus X.
   \end{equation*}
   There is a class $\beta \in H^2(G,B)$ which maps to $\alpha$ under Shapiro's isomorphism $\sh^2\colon H^2(G,B)\to H^2(G',M)$.
   Due to this the restriction $\res_G^{G'}(\beta)$ decomposes as $\alpha \oplus x \in H^2(G',M)\oplus H^2(G',X)$. 
   By Lemma~\ref{lem:diffuseclasses} the class $\res_G^{G'}(\beta)$ is diffuse.

   Let $\Gamma_0$ be the extension of $G$ which corresponds to the class $\beta$. By what we have seen, the subgroup
   $\Lambda_0 = \ker(\Gamma_0 \to G/G')$ is diffuse.
   Finally, we write the finite abelian group $G/G'$ as a quotient of a free abelian group $\ZZ^k$.
   By Lemma~\ref{lem:diffuseFibre} the fibre product $\Gamma_0 \times_{G/G'} \ZZ^k$ is diffuse.
   In fact, it is a Bieberbach group with holonomy group $G$.
   
   \end{proof}

   \subsection{Non-diffuse Bieberbach groups in small dimensions}\label{sec_lowDimensions}
   
   In this section we briefly describe the classification of all Bieberbach groups in dimension $d \leq 4$ which are not diffuse.
   The complete classification of crystallographic groups in these dimensions is given in \cite{BBNWZ-Classification} and we refer to them
   according to their system of enumeration.
   
   In dimensions $2$ and $3$ the classification is very easy. 
   In dimension $d=2$ there are two Bieberbach groups and both of them are diffuse. In dimension $d=3$ there are exactly $10$ Bieberbach groups. 
   The only group among those with vanishing first rational Betti number is the Promislow (or Hantzsche-Wendt) group $\Delta_P$ (which is called 3/1/1/04 in \cite{BBNWZ-Classification}).
   
   Now we consider the case $d=4$, in this case there are $74$ Bieberbach groups.
   As a consequence of the considerations for dimensions $2$ and $3$, a Bieberbach group $\Gamma$ of dimension $d = 4$ is not diffuse 
   if and only if it has vanishing Betti number or contains the Promislow group $\Delta_P$.
   Vanishing Betti number is something that can be detected easily from the classification.
   So how can one detect the existence of a subgroup isomorphic to $\Delta_P$?
   The answer is given in the following lemma.
   \begin{lemma}
     Let $\Gamma$ be a Bieberbach group acting on $E = \RR^4$ and assume that $b_1(\Gamma) > 0$.
     Let $\pi: \Gamma \to G$ be the projection onto the holonomy group.
     Then $\Gamma$ is not diffuse if and only if it contains elements $g,h \in \Gamma$ such that
     \begin{enumerate}[ (i)]
      \item $S := \langle \pi(g),\pi(h) \rangle \cong (\ZZ/2\ZZ)^2$,
      \item $\dim E^S = 1$ and
      \item if $E = E^S \oplus V$ as $S$-module, then $g\cdot0$ and $h\cdot 0$ lie in $V$.
     \end{enumerate}
   \end{lemma}
   \begin{proof}
    Since $b_1(\Gamma) >0$, the group $\Gamma$ is not diffuse exactly if it contains $\Delta_P$ as a subgroup.
    
    Assume $\Gamma$ contains $\Delta_P$, then $\Delta_P$ acts on a $3$-dimensional subspace $V$ and one can check (using $b_1(\Gamma) >0$) 
    that the holonomy of $\Delta_P$ embedds into $G$. The image will be denoted $S$.
    Take $g$ and $h$ in $\Gamma$ such $\pi(g)$ and $\pi(h)$ generate $S$, cleary $g\cdot0, h\cdot 0 \in V$.
    The group $S$ acts without non-trivial fixed points on $V$. Since $b_1(\Gamma) > 0$, thus $S$ acts trivially on a one-dimensional space $E^S$ and
    $E = E^S \oplus V$ is a decomposition as $S$-module.
    
    Conversely, if we can find $g,h \in \Gamma$ as above, then they generate a Bieberbach group of smaller dimension and with vanishing first Betti number.
    Hence they generate a group isomorphic to $\Delta_P$.
   \end{proof}

    Using this lemma and the results of the previous section one can decide for each of the $74$ Bieberbach groups whether they are diffuse or not.
    It turns out there are $16$ non-diffuse groups in dimension $4$, namely (cf.~\cite{BBNWZ-Classification}): 
    \begin{quote}
     04/03/01/006, 05/01/02/009, 05/01/04/006, 05/01/07/004, 06/01/01/049, 06/01/01/092, 06/02/01/027,
     06/02/01/050, 12/03/04/006, 12/03/10/005, 12/04/03/011, 13/04/01/023, 13/04/04/011, 24/01/02/004, 24/01/04/004, 25/01/01/010.
    \end{quote}
    The elementary abelian groups $(\ZZ/2\ZZ)^2$, $(\ZZ / 2\ZZ)^3$, the dihedral group $D_8$, the alternating group $A_4$ and
    the direct product group $A_4 \times \ZZ/2\ZZ$ occur as holonomy groups.
    Among these groups only four groups have vanishing first Betti number (these are 04/03/01/006, 06/02/01/027, 06/02/01/050 and 12/04/03/011).
    However, one can check that these groups contain the Promislow group as well. 
    In a sense the Promislow group is the only reason for Bieberbach groups in dimension $4$ to be non-diffuse
    (thus non of these groups has the unique product property).
    This leads to the following question:
    What is the smallest dimension $d_0$ of a non-diffuse Bieberbach group which does not contain $\Delta_P$?
    Clearly, such a group has vanishing first Betti number.
    Note that there is a group with vanishing first Betti number and holonomy $(\ZZ/3\ZZ)^2$ in dimension $8$ (see \cite{HillerSah1986}); thus $5\leq d_0 \leq 8$.
    The so-called generalized Hantzsche-Wendt groups are higher dimensional analogs of $\Delta_P$ (cf.~\cite{Szczepanski,RossettiSzczepanski}). 
    However, any such group $\Gamma$ with $b_1(\Gamma) = 0$ contains the Promislow group (see Prop.~8.2 in \cite{RossettiSzczepanski}).

\subsection{A family of non-diffuse infra-solvmanifolds}\label{sec_ExamplesOfSol}

Many geometric questions are not answered by the simple algebraic observation in Proposition \ref{prop_diffusePolycyclic}.
For instance, given a simply connected solvable Lie group $G$, is there an infra-solvmanifold of type $G$ with non-diffuse fundamental group?
To our knowlegde there is no criterion which decides whether a solvable Lie group $G$ admits a lattice at all. Hence we do not expect a simple 
answer for the above question.
We briefly discuss an infinite family of simply connected solvable groups where every infra-solvmanifold is commensurable to a non-diffuse one.
 
Let $\rho_1, \dots, \rho_n$ be $n\geq 1$ distinct real numbers with $\rho_i > 1$ for all $i=1,\dots,n$.
We define the Lie group
\begin{equation*}
     G := \RR^{2n} \rtimes \RR
\end{equation*}
where $s \in \RR$ 
acts by the diagonal matrix $\beta(s) := \diag(\rho_1^s, \dots, \rho_n^s,\rho_1^{-s},\dots,\rho_n^{-s})$ on $\RR^{2n}$. 
The group $G$ is a simply connected solvable Lie group. The isomorphism class of $G$ depends only one the line spanned by $(\log\rho_1,\ldots,\log\rho_n)$ in $\RR^n$.
For $n = 1$ the group $G$ is the three dimensional solvable group $\sol$, which will be reconsidered in Section \ref{3_mfd}.

\begin{prop}\label{prop_infrasolcom}
In the above setting the following holds.
\begin{itemize}
 \item[(a)]  The Lie group $G$ has a lattice if and only if 
  there is $t_0 > 0$ such that the polynomial $f(X) := \prod_{i=1}^n (1 - (\rho^{t_0}_i+\rho_i^{-t_0})X + X^2 )$ has integral coefficients.
 \item[(b)] If $G$ admits a lattice, then every infra-solvmanifold of type $G$ is commensurable to a non-diffuse one.
\end{itemize}
\end{prop}
Before we prove the proposition, we describe the group of automorphisms of $G$. Let $\sigma \in \aut(G)$, then
$\sigma(x,t) = (W x + f(t), \lambda t)$ for some $\lambda \in \RR^\times$, $W \in \GL_{2n}(\RR)$ and $f \in Z^1(\RR,\RR^{2n})$ a smooth cocycle for the action
of $s \in \RR$ on $\RR^{2n}$ via $\beta(\lambda\cdot s)$. 
Using that $H^1(\RR, \RR^{2n}) = 0$ we can compose $\sigma$ with an inner automorphism of $G$ (given by an element in $[G,G]$) such that
$f(t) = 0$.
Observe that the following equality has to hold 
\begin{equation*}
   \beta(\lambda t) W = W \beta(t)
\end{equation*}
for all $t \in \RR$. As a consequence $\lambda$ is $1$ or $-1$. In the former case $W$ is diagonal, in the latter case $W$ is a product of a diagonal
matrix and 
\begin{equation*}
W_0 = \begin{pmatrix}
       0 & 1_n \\
       1_n & 0 \\
      \end{pmatrix}.
\end{equation*}
Let $D_+$ denote the group generated by diagonal matrices in $\GL_{2n}(\RR)$ and $W_0$,  then $\aut(G) \cong \RR^{2n} \rtimes D_+$.

\begin{proof}[Proof of Proposition \ref{prop_infrasolcom}]
 Ad (a):
 Note that $N:= \RR^{2n} = [G,G]$ is the maximal connected normal nilpotent subgroup of $G$.
 Suppose that $G$ contains a lattice~$\Gamma$. Then $\Gamma_0 := \Gamma \cap N$ is a lattice in $N$ (cf.\ Cor.~3.5 in \cite{Raghunathan1972}) and
 $\Gamma / \Gamma_0$ is a lattice in $G/N \cong \RR$.
 Let $t_0 \in \RR$ so that we can identify $\Gamma / \Gamma_0$ with $\ZZ t_0$ in $\RR$.
 Take a basis of $\Gamma_0$, w.r.t.\ this basis $\beta(t_0)$ is a matrix in $\SL_{2n}(\ZZ)$.
 The polynomial $f$ is the charcteristic polynomial of $\beta(t_0)$ and the claim follows.
 
 Conversely, let $t_0 > 0 $ with $f \in \ZZ[X]$ as above.
 Take any matrix $A \in \SL_{2n}(\ZZ)$ with characteristic polynomial $f$, e.g.~if $f(X) = X^{2n} + a_{2n-1}X^{2n-1} + \dots + a_1 X + a_0$ then
 the matrix $A$ with ones above the diagonal and last row $(-a_0, -a_1, \dots, -a_{2n-1})$ has suitable characteristic polynomial.
 
 Since by assumption all the $\rho_i$ are distinct and real,
 we find $P \in \GL_{2n}(\RR)$ with
 $PAP^{-1} = \beta(t_0)$. Now define $\Gamma_0 := P \ZZ^{2n}$ and we obtain a lattice $\Gamma := \Gamma_0 \rtimes (\ZZ t_0)$ in $G$.

 \medskip
 
 Ad (b):
 Let $\Lambda \subset \aff(G)$ be the fundamental group of an infra-solvmanifold.
 Define $\Gamma := G \cap \Lambda$ and $\Gamma_0 := \Gamma \cap N$ where $N = \RR^{2n}$ is the maximal normal nilpotent subgroup.
 The first Betti number of $\Lambda$ is $b_1(\Lambda) = \dim_{\RR} (G/N)^{\Lambda/\Gamma}$. 
 The quotient $\Gamma/\Gamma_0$ is a lattice in $\RR$, so is of the form $\ZZ t_0$ for some $t_0 > 0$.
 
 Take any basis of the lattice $\Gamma_0 \subseteq \RR^{2n}$. We shall consider coordinates on $\RR^{2n}$ w.r.t.\ this basis from now on.
 In particular, $\beta(t_0)$ is given by an integral matrix $A \in \SL_{2n}(\ZZ)$ and further 
 $\Gamma$ is isomorphic to the strongly polycyclic group $\ZZ^{2n} \rtimes \ZZ$ where $\ZZ$ acts via $A$.
 Let $F/ \QQ$ be a finite totally real Galois extension which splits the characteristic polynomial of $A$, so the Galois group
 permutes the eigenvalues of $A$.
 Moreover, the Galois group acts on $\Gamma_0 \otimes_\ZZ F$ so that we can find a set of eigenvectors which are permuted accordingly.
 Let $B \in \GL_{2n}(F)$ be the matrix whose columns are the chosen eigenvectors, then
 $B^{-1} A B = \beta(t_0)$ and for all $\sigma \in \gal(F/\QQ)$ we have $\sigma(B) = B P_\sigma$ for a permutation matrix $P_\sigma \in \GL_{2n}(\ZZ)$.
 It is easily seen that $P_\sigma$ commutes with $W_0$, and hence $W = B W_0 B^{-1}$ is stable under the Galois group, this means $W \in  \GL_{2n}(\QQ)$.
 
 Since $W$ is of order two, we can find a sublattice $L \subset \Gamma_0$ which admits a basis of eigenvectors of $W$. Pick one of these eigenvectors, say $v$,
 with eigenvalue one, find $q \in \ZZ\setminus\{0\}$ with $q \Gamma_0 \subset L$ and
 take a positive integer $r$ so that
 \begin{equation*}
    A^r \equiv 1 \bmod 4q.
 \end{equation*}
 Like that we find a finite index subgroup $\Gamma' := L \rtimes r\ZZ$ of $\Gamma$ which is stable under the automorphism $\tau$
 defined by $(x,t) \mapsto (W x,-t)$.
 Since we want to construct a torsion-free group we cannot add $\tau$ into the group. Instead we take the group $\Lambda'$ generated by
 $(\frac{1}{2} v, 0) \tau$ and $\Gamma'$ in the affine group $\aff(G)$. Then $\Lambda'$ is the fundamental group of an infra-solvmanifold of type $G$ which is commensurable 
 with $\Lambda$.
\end{proof}

%%%%%%%%%%%%%%%%%%%%%%%%%%%%%%%%%%%%%%%%%%%%%%%%%%%%%%%%%%%%%%%%%%%%%%%%%%%%%%%%

\section{Fundamental groups of hyperbolic manifolds}

\label{sec_rank1}

In this section we prove Theorems \ref{rank1_lattice:intro} and \ref{geomfin_dim3:intro} from the introduction. We give a short overview of rank one symmetric spaces before studying first their unipotent and then their axial isometries in view of applying Lemma \ref{dist_move_crit}. Then we review some well-known properties of geometrically finite groups of isometries before proving a more general result (Theorem \ref{rank1}) and showing how it implies Theorems \ref{rank1_lattice:intro} and \ref{geomfin_dim3:intro}. We also study the action on the boundary, resulting in Theorem \ref{real_bd_diff}, which will be used in the next section.

%%%%%%%%%%%%%%%%%%%%%%%%%%%%%%%%%%%%%%%%%%%%%%%%%%%%%%%%%%%%

\subsection{Hyperbolic spaces}

\subsubsection{Isometries}

We recall some terminology about isometries of Hadamard manifolds: if $g\in\isom^+(X)$ where $X$ is a complete simply connected manifold with non-positive curvature then $g$ is said to be
\begin{itemize}
\item Hyperbolic (or axial) if $\min(g) = \inf_{x\in X} d_X(x,gx) > 0$; 
\item Parabolic if it fixes exactly one point in the visual boundary $\pl X$, equivalently $\min(g)=0$ and $g$ has no fixed point inside $X$. 
\end{itemize}
We will be interested here in the case where $X=G/K$ is a symmetric space associated to a simple Lie group $G$ {\it of real rank one}. An element $g\in G$ then acts on $X$ as an hyperbolic isometry if and only if it is semisimple and has an eigenvalue of absolute value $>1$ in the adjoint representation. Parabolic isometries of $X$ are algebraically characterised as corresponding to the non-semisimple elements of $G$; their eigenvalues are necessarily of absolute value one. If they are all equal to one then the element of $G$ is said to be unipotent, as well as the corresponding isometry of $X$.

%%%%%%%%%%%%%%%%%%%%%%%%%%%%%%

\subsubsection{Projective model}

Here we describe models for the hyperbolic spaces $\HH_A^n$ for $A=\RR,\CC,\HH$ (the symmetric spaces associated to the Lie groups $\SO(n,1), \, \SU(n,1)$ and $\Sp(n,1)$ respectively) which we will use later for computations. We will denote by $z\mapsto\ovl z$ the involution on $A$ fixing $\RR$, and define as usual the reduced norm and trace of $A$ by 
$$
|z|_{A/\RR} = \ovl z z = z\ovl z, \quad \tr_{A/\RR}(z) = z +\ovl z
$$
We let $V=A^{n,1}$, by which we mean that $V$ is the right $A$-vector space $A^{n+1}$ endowed with the sesquilinear inner product given by\footnote{We use the model of \cite{Kim_Parker} rather than that of \cite{Phillips_dir}.} 
$$
\langle v, v' \rangle  = \overline{v_{n+1}'}v_1 + \sum_{i=2}^n \overline{v_i'}v_i + \overline{v_1'}v_{n+1}. 
$$
The (special if $A=\RR$ or $\CC$) isometry group $G$ of $V$ is then isomorphic to $\SO(n,1)$, $\SU(n,1)$ or $\Sp(n,1)$. Let: 
$$
V_- = \{v\in V\: | \: \langle v,v\rangle < 0\} = \left\{ v\in V \:|\: \tr_{A/\RR}(\overline{v_1}v_{n+1}) < -\sum_{i=2}^n |v_i|_{A/\RR} \right\}
$$
then the image $X = \PP V_-$ of $V_-$ in the $A$-projective space $\PP V$ of $V$ can be endowed with a distance function $d_X$ given by:
\begin{equation}
\cosh \left( \frac{d_X([v],[v'])}2 \right)^2 = \frac{|\langle v,v'\rangle|_{A/\RR}}{\langle v,v\rangle \langle v',v'\rangle}. 
\label{dist}
\end{equation}
This distance is $G$-invariant, and the stabilizer in $G$ of a point in $V_-$ is a maximal compact subgroup of $G$. Hence the space $X$ is a model for the symmetric space $G/K$ (where $K=\SO(n), \SU(n)$ or $\Sp(n)$ according to whether $A=\RR,\CC$ or $\HH$). 

The following lemma will be of use later. 

\begin{lemma}
If $v,v'\in V_-$ then $\tr_{A/\RR}(\overline{v_{n+1}}v_{n+1}'\langle v,v' \rangle) < 0$. 
\label{trace_neg}
\end{lemma}

\begin{proof}
Since $\tr_{A/\RR}(\overline{v_{n+1}}v_{n+1}'\langle v,v' \rangle)$ does not change sign when we multiply $v$ or $v'$ by a element of $A$ from the right, we may suppose that $v_{n+1}=v_{n+1}'=1$. In this case we have:
$$
\tr_{A/\RR}(\overline{v_{n+1}}v_{n+1}'\langle v,v'\rangle) = \tr_{A/\RR}(v_1) + \tr_{A/\RR}(v_1') + \tr_{A/\RR}\left( \sum_{i=2}^n \overline{v_i'} v_i \right). 
$$
Now we have 
$$
\tr_{A/\RR}\left( \sum_{i=2}^n \overline{v_i'} v_i \right) \le 2\sqrt{\left(\sum_{i=2}^n |v_i|_{\A/\RR} \right)\cdot\left(\sum_{i=2}^n |v_i'|_{\A/\RR} \right)}
$$
by Cauchy-Schwarz, and since $v,v'\in V_-$ we get
\begin{align*}
\tr_{A/\RR}(\overline{v_{n+1}}v_{n+1}'\langle v,v'\rangle) 
           &< \tr_{A/\RR}\left( \sum_{i=2}^n \overline{v_i'} v_i \right) - \sum_{i=2}^n |v_i|_{\A/\RR} - \sum_{i=2}^n |v_i'|_{\A/\RR} \\
           &\le -\left( \sqrt{\sum_{i=2}^n |v_i|_{\A/\RR}} - \sqrt{\sum_{i=2}^n |v_i'|_{\A/\RR}}\right)^2 \le 0. 
\end{align*}
\end{proof}

%%%%%%%%%%%%%%%%%%%%%%%%%%%%%%%%%%%%%%%%%%%%%%%%%%%%%%%%%%%%

\subsection{Unipotent isometries and distance functions}
\label{unipotent}

In this subsection we prove the following proposition, which is the main ingredient we use in extending the results of \cite{Bowditch_diff} from cocompact subgroups to general lattices. 

\begin{prop}
Let $A$ be one of $\RR,\CC$ or $\HH$ and let $\eta\neq 1$ be a unipotent isometry of $X = \HH_A^n$ and $a,x\in\HH_A^n$. Then 
$$
\max\bigl(d(a,\eta x),d(a,\eta^{-1}x)\bigr) > d(a,x). 
$$
\label{parsep}
\end{prop}

\begin{proof}
We say that a function $h:\ZZ\to\RR$ is strictly convex if $h$ is the restriction to $\ZZ$ of a strictly convex function on $\RR$ (equivalently all points on the graph of $h$ are extremal in their convex hull and $h$ has a finite lower bound). We will use the following criterion, similar to Lemma 6.1 in \cite{Phillips_dir}.

\begin{lemma}
Let $X$ be a metric space, $x\in X$ and let $\phi\in\isom(X)$. Suppose that there exists an increasing function $f:[0,+\infty[\to\RR$ such that for any $y\in X$ the function $h_y : k\mapsto f(d_X(y,\phi^k x))$ is strictly convex. Let 
$$
B_k = \{y\in X:\: d_X(y,\phi^k x) \le d_X(y,x)\}.
$$
Then we have $B_1\cap B_{-1} = \emptyset$. 
\end{lemma}

\begin{proof}
Suppose there is a $y\in X$ such that 
$$
d_X(y,\phi x), d_X(y,\phi^{-1} x) \le d_X(y,x). 
$$
Since $f$ is increasing this means that $h_y(1),h_y(-1) \le h_y(0)$: but this is impossible since $h_y$ is strictly convex. 
\end{proof}

Applying it to $\phi=\eta$, we see that it suffices to prove that for any $z,w\in X$ the function
$$
f :\: t\in\RR \mapsto \cosh\left( \frac{d_X(z, \eta^t w)} 2\right)^2
$$
is strictly convex on $\RR$, i.e. $f'' > 0$. Of course we need only to prove that $f''(0) > 0$ since $z,w$ are arbitrary.
By the formula \eqref{dist} for arclength in hyperbolic spaces it suffices to prove this for the function
$$
h :\: t\mapsto |\langle v,\eta^t v' \rangle|_{A/\RR}
$$
for any two $v,v'\in A^{n,1}$ (which we normalize so that their last coordinate equals~1). Now we have:
\begin{align*}
\frac{d^2h}{dt^2} &= \frac d{dt}\left( \tr_{A/\RR} \left( \overline{\langle v,\eta^t v' \rangle} \frac d{dt} \langle v,\eta^t v' \rangle \right)\right) \\
                 &= 2\left| \frac d{dt} \langle v,\eta^t v' \rangle \right|_{A/\RR} + \tr_{A/\RR} \left( \overline{\langle v,\eta^t v' \rangle} \frac {d^2}{dt^2} \langle v,\eta^t v' \rangle \right). 
\end{align*}
There are two distinct cases (see either \cite[Section 3]{Phillips_dir} or \cite[Section 1]{Kim_Parker}): $\eta$ can be conjugated to a matrix of either one of the following forms:
$$
\begin{pmatrix} 1 & -\ovl a & -|a|_{A/\RR}/2 \\
                0 &   1_{n-1}   &      a          \\
                0 &    0    &      1     \end{pmatrix},\, a\in A^n
\text{ or }
\begin{pmatrix} 1 &  0  & b \\
                0 & 1_{n-1} & 0 \\
                0 &  0  & 1 \end{pmatrix},\, b\in A \text{ totally imaginary.}
$$
In the second case we get that $\frac {d^2}{dt^2} \eta^t = 0$, hence 
$$
\frac{d^2h}{dt^2} = 2 \left| \frac d{dt} \langle v,\eta^t v' \rangle \right|_{A/\RR} = 2 |\ovl{b}|_{A/\RR} > 0.
$$
In the first case (which we normalize so that $|a|_{A/\RR} = 1$) we have at $t=0$: 
$$
\frac{d^2h}{dt^2} = 2\left| \frac d{dt} \langle v,\eta^t v' \rangle \right|_{A/\RR} - \tr_{A/\RR}\left(\overline{v_{n+1}}v_{n+1}' \langle v,v' \rangle\right)
$$
and hence the result follows from Lemma \ref{trace_neg}. 
\end{proof}

%%%%%%%%%%%%%%%%%%%%%%%%%%%%%%%%%%%%%%%%%%%%%%%%%%%%%%%%%%%%

\subsection{Hyperbolic isometries and distance functions}
\label{hyperbolic}

Axial isometries in negatively curved spaces have in some regards a much simpler behaviour than parabolic ones: to establish separation of bisectors one only needs to use the hyperbolicity of the space on which they act (as soon as their minimal displacement is large enough), as was already observed in \cite{Bowditch_diff} (see Lemma \ref{minsep} below). On the other hand, isometries with small enough minimal displacement which rotate non-trivially around their axis obviously do not have the bisector separation property; we study this phenomenon in more detail for real hyperbolic spaces below, obtaining an optimal criterion in Proposition \ref{bowditch_mieux}.  

\subsubsection{Gromov-hyperbolic spaces}

The following lemma is a slightly more precise version of Corollary 5.2 \cite{Bowditch_diff}. It has essentially the same proof;
we will give the details, which are not contained in \cite{Bowditch_diff}. 

\begin{lemma}
Let $\delta>0$ and $d>0$; there exists a constant $C(\delta,d)$ such that for any
$\delta$-hyperbolic space $X$ and any axial isometry $\gamma$ of $X$ such that $\min(\gamma)\ge C(\delta,d)$ and any pair $(x,a)\in X$ we have
$$
\max(d(\gamma x,a),d(\gamma^{-1} x, a)) \ge d(x,a) + d. 
$$
\label{minsep}
\end{lemma}

\begin{proof}
Let $\gamma$ be as in the statement (with the constant $C = C(\delta,d)$ to be determined later), let $L$ be its axis. Let $w,w',w''$ be the projections 
of $x,\gamma x,\gamma^{-1} x$ on $L$, and $v$ that of $a$. We will suppose (without loss of generality) that $v$ lies on the ray in $L$ originating at $w$ and passing through $w'$. 

Now let $T$ be a metric tree with set of vertices constructed as follows: we take the geodesic segment on $L$ containing all of $w,w',w''$ and $v$ and we add the arcs $[x,w]$, etc. Then, for any two vertices $u,u'$ of $T$ we have
$$
d_X(u,u') \le d_T(u,u') \le d_X(u,u') + c
$$
where $c$ depends only on $\delta$ (see the proof of Proposition 6.7 in \cite{Bowditch_ggt}). In this tree we have
\begin{align*}
d_T(a,\gamma^{-1}x) &= d_X(a,v) + d_X(v,w) + d_X(w,w'') + d_X(w'',\gamma^{-1}x) \\
     &= d_X(w,\gamma^{-1} w) + d_X(a,v) + d_X(v,w) + d_X(w,x) \\
     &= \min(\gamma) + d_T(a,x)
\end{align*}
and using both inequalities above we get that
$$
d_X(a,\gamma^{-1}x) \ge d_T(a,\gamma^{-1}x) - c \ge d_X(a,x) + \min(\gamma) - c.
$$
We see that for $\min(\gamma) \ge C(\delta,d) = c + d$ the desired result follows. 
\end{proof}

%%%%%%%%%%%%%%%%%%%%%%%%%%%%%%

\subsubsection{A more precise result in real hyperbolic spaces}

We briefly discuss a quantitative version of Lemma \ref{minsep}.
Bowditch observed (cf.\ Thm.~5.3 in \cite{Bowditch_diff}) that a group $\Gamma$ which acts freely by axial transformations on
the hyperbolic space $\HH_\RR^n$ is diffuse if every $\gamma \in \Gamma \setminus\{1\}$ has
translation length at least $2\log(1+\sqrt{2})$. 
We obtain a slight improvement relating the lower bound on the translation length more closely to the eigenvalues
of the rotational part of the transformation.

Our proof is based on a calculation in the upper half-space model of $\HH_{\RR}^n$, i.e.\ we consider $\HH_\RR^n = \{\: x \in \RR^n \:|\: x_n > 0 \:\}$ with the hyperbolic metric $d$
(see \S 4.6 in \cite{Ratcliffe}). 
Every axial transformation $\gamma$ on $\HH_\RR^n$ is conjugate to a transformation of the form $x \mapsto kA x$ where $A$ is an orthogonal matrix in $\OO(n-1)$ (acting on the first $n-1$ components) 
and $k>1$ is a real number (see Thm.~4.7.4 in \cite{Ratcliffe}).
We say that $A$ is the \emph{rotational part} of $\gamma$. The translation length of $\gamma$ is given by $\min(\gamma)=\log(k)$.
We define the \emph{absolute rotation} $r_\gamma$ of $\gamma$ to be the maximal value of $|\lambda-1|$ where $\lambda$ runs through all eigenvalues of $A$.
In other words, $r_\gamma$ is merely the operator norm of the matrix $A - 1$. 
The absolute rotation measures how close the eigenvalues get to $-1$. 
It is apparent from Bowditch's proof that the case of eigenvalue $-1$ (rotation of angle $\pi$) is
the problematic case whereas the situation should improve significantly for rotation bounded away from angle~$\pi$. We prove the following sharp result. 

\begin{prop}
\label{bowditch_mieux}
 An axial transformation $\gamma$ of $\HH_\RR^n$ has the property
  \begin{equation}\tag{$\bigstar$}\label{eq:propertyMax}
    \max( d(x,\gamma y), d(x,\gamma^{-1} y) ) > d(x,y) \text{ for all } x,y \in \HH_\RR^n
  \end{equation}
if and only if the translation length $\min(\gamma)$ satisfies
  \begin{equation}\tag{$\clubsuit$}\label{eq:lowerbound}
      \min(\gamma) \geq \arcosh(1 + r_\gamma).
  \end{equation}
\end{prop}

Using the same argument as above we immediately obtain the following improvement of Bowditch's Theorem 5.3 (we use Proposition \ref{parsep} to take care of the unipotent elements). 

\begin{corollary*}
  Let $\Gamma$ be a group which acts freely by axial or unipotent transformations of the hyperbolic space $\HH_\RR^n$.
  If the translation length of every axial $\gamma \in \Gamma$ satisfies inequality \eqref{eq:lowerbound}, then $\Gamma$ is diffuse.
\end{corollary*}

\begin{remark}
 (1) It is a trivial matter to see that the converse of the corollary does not hold. Take any axial transformation $\gamma \neq 1$ which does not obey inequality \eqref{eq:lowerbound}, then the diffuse group $\Gamma = \ZZ$ acts via $\gamma$ on $\HH^n$.

(2) If $\gamma\in\SL_2(\CC)$ is hyperbolic, with an eigenvalue $\lambda = e^{\ell/2}e^{i\theta/2}$ then the condition \eqref{eq:lowerbound} is equivalent to 
\begin{equation*}    %\tag{$\clubsuit'$}\label{eq:lowerbound_dim3}  %%% we did never refer here
  \cosh(\ell) \ge 1 + \sqrt{2 - 2\cos(\theta)}. 
\end{equation*}
\end{remark}

\begin{proof}[Proof of Proposition \ref{bowditch_mieux}]
 Let $\gamma$ be an axial transformation which satisfies~\eqref{eq:lowerbound}.
 We will show that for all $x,y \in \HH_\RR^n$ we have $\max( d(x,\gamma y), d(x,\gamma^{-1} y) ) > d(x,y)$.
 After conjugation we can assume that $\gamma(a) = Ak a$ with $k>1$ and $A \in \OO(n-1)$.
 We take $x,y$ to lie in the upper half-space model, then we may consider them as elements of $\RR^n$. 
 We will suppose in the sequel that $\Vert x \Vert \leq \Vert y \Vert$  in the euclidean metric of $\RR^n$,
 and under this hypothesis we shall prove that $d(x,\gamma y) > d(x,y)$. If the opposite inequality 
 $\|x\|\ge\|y\|$ holds we get that $d(y,\gamma x) > d(x,y)$, hence $d(x,\gamma^{-1} y) > d(x,y)$ which implies the proposition. 
 
 Using the definition of the hyperbolic metric and the monotonicity of $\cosh$ on positive numbers, it suffices to show
 \begin{equation*}
       \Vert x - Aky \Vert^2 > k \Vert x- y \Vert^2.
 \end{equation*}
 In other words, we need to show that the largest real zero of the quadratic function
 \begin{equation*}
 f(t) = t^2 \Vert y \Vert^2 - t(\Vert x \Vert^2 + \Vert y \Vert^2 + 2 \langle x , Ay -y\rangle) + \Vert x \Vert^2
 \end{equation*}
 is smaller than $\exp(\arcosh(1+r_\gamma)) = 1+ r_\gamma + \sqrt{r_\gamma^2 + 2 r_\gamma}$.
 We may divide by $\Vert y \Vert^2$ and we can thus assume $\Vert y \Vert = 1$ and $0 <\Vert x \Vert \leq 1$.
 The large root of $f(t)$ is
 \begin{equation*}
    t_0 = \frac{\Vert x \Vert^2 + 1}{2} + \langle x , Ay -y\rangle + \frac{1}{2}\sqrt{(\Vert x \Vert^2 + 1 + 2\langle x , Ay -y\rangle)^2 - 4 \Vert x \Vert^2}.
 \end{equation*}
 Note that if $r_\gamma = 0$, then $k > 1 = t_0$.
 
 Suppose that $r_\gamma > 0$.
 Indeed, by Cauchy-Schwarz
 $|\langle x, Ay-y \rangle| < r_\gamma \Vert x \Vert$ and the inequality is strict since $x_n > 0$.
 As a consequence $t_0 < t(\Vert x \Vert)$ where
 \begin{equation*}
   t(s) = \frac{s^2 + 1}{2} + r_\gamma s + \frac{1}{2}\sqrt{(s^2 + 1 + 2 r_\gamma s)^2 - 4 s^2}.
 \end{equation*}
 Finally, we determine the maximum of the function $t(s)$ for $s\in [0,1]$.
 A simple calculation shows that there is no local maximum in the interval $[0,1]$.  
 We conclude that the maximal value is attained at $s = 1$ and is precisely
 \begin{equation*}
 t(1) = 1 + r_\gamma +\sqrt{r_\gamma^2 + 2 r_\gamma}.
 \end{equation*}
 
  Conversely, assume that \eqref{eq:lowerbound} does not hold. In this case we have $1 < k < 1 + r_\gamma +\sqrt{r_\gamma^2 + 2 r_\gamma}$ and thus $r_\gamma \neq 0$.
  Choose some vector $y \in \RR^n$ with $y_n = 0$ and $\Vert y \Vert = 1$ so that $\Vert Ay-y\Vert = r_\gamma$ (this is possible since $r_\gamma$ is the operator norm of $A-1$).
  We define $x = r_\gamma^{-1} (Ay - y)$ and we observe that $x \neq y$ since the orthogonal matrix $A$ has no eigenvalues of abolute value exceeding one.
  The following inequalities hold:
  \begin{equation*}
   \frac{\Vert x - k^{-1}A^{-1} y \Vert^2}{k^{-1}} \leq \frac{\Vert x - kA y \Vert^2}{k} < \Vert x - y \Vert^2.
  \end{equation*}
  The first follows from $\langle x, A^{-1}y \rangle \leq \langle x, y \rangle + r_\gamma  = \langle x, Ay \rangle$. The second inequality follows from
  the assumption $k < 1 + r_\gamma +\sqrt{r_\gamma^2 + 2 r_\gamma}$.
  Since the last inequality is strict, we can use continuity to find distinct $x'$ and  $y'$ in the upper half-space (close to $x$ and $y$), so that still
  \begin{equation*}
   \max\left\{\frac{\Vert x' - k^{-1}A^{-1} y' \Vert^2}{k^{-1}}, \frac{\Vert x' - kA y' \Vert^2}{k}\right\} < \Vert x' - y' \Vert^2.
  \end{equation*}
  Interpreting $x'$ and $y'$ as points in the hyperbolic space, the assertion follows from the definition of the hyperbolic metric.
\end{proof}

%%%%%%%%%%%%%%%%%%%%%%%%%%%%%%%%%%%%%%%%%%%%%%%%%%%%%%%%%%%%

\subsection{Geometric finiteness}

There are numerous equivalent definitions of geometric finiteness for discrete subgroups of isometries of rank one spaces, see for example \cite[Section 3.1]{Matsuzaki_Taniguchi} or \cite[Section 12.4]{Ratcliffe} for real hyperbolic spaces. We shall use the equivalent definitions given by B. Bowditch in \cite{Bowditch_geom_fin} for general negatively-curved manifolds. 

The only facts from the theory of geometrically finite groups we will need in this section are the following two lemmas which are quite immediate consequences of the equivalent definitions. 
 
\begin{lemma}
Let $G$ be a rank-one Lie group and $\Gamma\le G$ be a geometrically finite subgroup, all of whose parabolic elements have finite-order eigenvalues. Then there is a subgroup $\Gamma'\le\Gamma$ of finite index such that all parabolic isometries contained in $\Gamma'$ are unipotent elements of $G$. 
\label{nofun}
\end{lemma}

\begin{proof}
From \cite[Corollary 6.5]{Bowditch_geom_fin} we know that $\Gamma$ has only finitely many conjugacy classes of maximal parabolic subgroups;
by residual finiteness of $\Gamma$ we will be finished if we can show that for any parabolic subgroup $P$ of $G$ such that the fixed point of $P$ in $\pl\HH_\RR^n$ is a cusp point, the group $\Lambda=\Gamma\cap P$ is virtually unipotent.
Writing $P=MAN$ where $A$ is a split torus, $M$ compact and $N$ the unipotent radical we see that it suffices to verify that the projection 
of $\Lambda$ on $A$ is trivial (Indeed, since then $\Lambda$ is contained in $MN$, and its projection to $M$ is finite because
it has only finite-order elements by the hypothesis on eigenvalues, and it is finitely generated by \cite[Proposition 4.1]{Bowditch_geom_fin}).
This follows from discreteness of $\Gamma$: if it contained an element $\lambda$ with a non-trivial projection on $A$, then for any non-trivial $n\in N$ we have that either $\lambda^k n\lambda^{-k}$ or $\lambda^{-k} n\lambda^k$ goes to the identity of $G$; but since the fixed point of $P$ is a cusp point 
for $\Gamma$ the intersection $\Gamma\cap N$ must be nontrivial, hence there cannot exist such a $\lambda$. 
\end{proof}

\begin{lemma}
Let $G$ be a rank-one Lie group, $\Gamma$ a torsion-free geometrically finite subgroup of $G$ and $M_\Gamma=\Gamma\bs X$. Then for any $\ell_0$ there are only finitely many closed geodesics of length less than $\ell_0$ in $M_\Gamma$. 
\label{fin_geod}
\end{lemma}

\begin{proof}
See also \cite[Theorem 12.7.8]{Ratcliffe}.
One of Bowditch's characterizations of geometrical finiteness is the following:
 let $L_\Gamma\subset\pl X$ be the limit set of $\Gamma$,
i.e. the closure of the set of points fixed by some nontrivial element of $\Gamma$, and let $Y_\Gamma\subset X$ be the convex hull in $X$ of $L_\Gamma$. Let $C_\Gamma = \Gamma\bs Y_\Gamma$ (the `convex core' of $M_\Gamma$), and let $M_{[\eps,+\infty[}$ be the $\eps$-thick part of $M_\Gamma$. Then $\Gamma$ is geometrically finite if and only if $C_\Gamma\cap M_{[\eps,+\infty[}$ is compact (for some or any $\eps$): see \cite[Section 5.3]{Bowditch_geom_fin}. 

It is a well-known consequence of 
Margulis' lemma
that there is an $\eps_0>0$ such that all geodesics in $M_\Gamma$ of length less than $\ell_0$ are contained in the $\eps_0$-thick part. On the other hand it is clear that any closed geodesic of $M_\Gamma$ is contained in $C_\Gamma$ (since the endpoints of any lift are in $L_\Gamma$) and hence all closed geodesics of $M_\Gamma$ with length $\le\ell_0$ are contained in the compact set $C_\Gamma\cap M_{[\eps_0,+\infty[}$, 
which implies that there are only finitely many such. 
\end{proof}

%%%%%%%%%%%%%%%%%%%%%%%%%%%%%%%%%%%%%%%%%%%%%%%%%%%%%%%%%%%%

\subsection{Main results}

\subsubsection{Action on the space}

\begin{theorem}
Let $G$ be one of the Lie groups $\SO(n,1)$, $\SU(n,1)$ or $\Sp(n,1)$, $X$ 
the associated symmetric space and let $\Gamma$ be a geometrically finite subgroup
of $G$. Suppose that all eigenvalues of parabolic elements of $\Gamma$ are roots
of unity.
Then there exists a finite-index subgroup $\Gamma'\subset\Gamma$ such that $\Gamma'$ acts diffusely on~$X$. 
\label{rank1}
\end{theorem}

\begin{proof}
Let $\Gamma'$ be a finite-index subgroup of $\Gamma$ such that all semisimple elements 
$\gamma\in\Gamma'$ have $\min(\gamma)>C(\delta_X,1)$ (where $\delta_X$ is a hyperbolicity constant for $X$, which is Gromov-hyperbolic since it is a negatively-curved, simply connected Riemannian manifold)---such a subgroup
exists by Lemma \ref{fin_geod} and the residual finiteness of $\Gamma$.
By Lemma \ref{nofun} we may also suppose that the parabolic isometries in $\Gamma'$ are exclusively unipotent.

Now we can check that the hypothesis \eqref{dist_move_crit_eq} in Lemma \ref{dist_move_crit} holds for the action of $\Gamma$ on $X$: for axial isometries we only have to apply Lemma \ref{minsep}, and for unipotent elements Proposition \ref{parsep}. 
\end{proof}

The hypothesis on eigenvalues of parabolic elements is equivalent to asking that every parabolic subgroup of $\Gamma$ contains a finite-index subgroup which consists of unipotent elements.
It is necessary for an application of Lemma \ref{dist_move_crit}, as shown by the following construction. 

\begin{lemma}
For $n\ge 4$ there exists a discrete, two-generated free subgroup $\Gamma$ of $\SO(n,1)$ such that for all $x\in \HH_\RR^n$ there is a $y\in \HH_\RR^n$ and a $g\in\Gamma\setminus\{1\}$ such that 
$$
d(x,y) \ge d(gx,y),\, d(g^{-1}x,y). 
$$
\end{lemma}

\begin{proof}
It suffices to prove this lemma for $\SO(4,1)$. Let $\omega$ be an infinite-order rotation of $\RR^2$ and let $\phi$ be the isometry of $\RR^3=\RR\times\RR^2$ given by $(t,x)\mapsto (t+1,\omega\cdot x)$. Then it is easy to see that for any $k$ and any $x$ not on the axis $\RR\times 0$ of $\phi$ the bisectors between $x$ and $\phi^{\pm k}x$ intersect. Let $\wdt\phi$ be the isometry of $\HH_\RR^4$ obtained by taking the Poincar\'e extension of $\phi$ (i.e. we fix a point on $\pl\HH_\RR^4$ and define $\wdt\phi$ by identifying the horospheres at this point with the Euclidean three--space on which $\phi$ acts), which will also not satisfy \eqref{dist_move_crit_eq} for all points outside of a two dimensional totally geodesic submanifold $Y_\phi$. 

Now take $\phi_1,\phi_2$ as above. There exists a $g\in\isom(\HH_\RR^4)$ such that $gY_{\phi_2}g^{-1}\cap Y_{\phi_2} = \emptyset$, and then for any $k_1,k_2>0$ the group $\langle \wdt\phi_1^{k_1},\wdt\phi_2^{k_2}\rangle$ satisfies the second conclusion of the lemma. It remains to prove that for $k_1,k_2$ large enough it is a discrete (and free) group. This is done by a very standard argument which goes as follows: There are disjoint open neighbourhoods $U_i$ of $\fix(\wdt\phi_i)$ in $\pl\HH_\RR^4$ (not containg $\fix(\wdt\phi_j),j\not=i$) and positive integers $k_1,k_2$ such that for all $k\in\ZZ,|k|\ge k_i$ we have $\wdt\phi_i^k(\HH_\RR^4\setminus U_i)\subset U_i$. Now we can apply the ping-pong lemma of Klein to obtain freeness and discreteness of $\langle \wdt\phi_1^{k_1},\wdt\phi_2^{k_2}\rangle$: fix a $\xi\in\pl\HH_\RR^4\setminus(U_1\cup U_2)$, then any non-trivial reduced word in $\wdt\phi_1,\wdt\phi_2$ sends $\xi$ inside one of $U_1$ or $U_2$, hence the orbit of $\xi$ is discrete in $\pl\HH_\RR^4$ (
proving discreteness of $\langle \wdt\phi_1^{k_1},\wdt\phi_2^{k_2}\rangle$) and any such word is nontrivial in $\SO(4,1)$ (proving freeness). 
\end{proof}

On the other hand this phenomenon cannot happen in $\HH_\RR^2,\HH_\RR^3$, which 
yields the following corollary of Theorem \ref{rank1}. 

\newtheorem{corollary}{Corollary}

\begin{corollary}
\label{klein3}
If $\Gamma$ is a finitely generated discrete subgroup of $\SL_2(\CC)$ then $\Gamma$ is virtually diffuse. 
\end{corollary}

\begin{proof}
Since in dimension three all Kleinian groups are isomorphic to geometrically 
finite ones (this is a consequence of Thurston's hyperbolization theorem for Haken manifolds, as explained in \cite[Theorem 4.10]{Matsuzaki_Taniguchi})
the result would follow if we can prove diffuseness for the latter class.
But parabolic isometries of $\HH^3$ are necessarily unipotent (since if an element of $\SL_2(\CC)$ has two equal eigenvalues,
they must be equal to $\pm 1$, and hence it is unipotent in the adjoint representation), and thus we can
apply Theorem \ref{rank1} to deduce that a geometrically finite Kleinian group in dimension three has a finite-index subgroup which acts diffusely on $\HH^3$.

We could also deduce Corollary \ref{klein3} from the veracity of the Tameness conjecture \cite{Agol_tame}, \cite{Calegari_Gabai} and the virtual diffuseness of three--manifolds groups, Theorem \ref{res_dim3} from the introduction. 
\end{proof}

Also, when parabolic subgroups of $\Gamma$ are large enough\footnote{For example, in the real hyperbolic case, when their span in the Lie algebra is of codimension smaller than one.} the hypothesis should be satisfied. We will be content with the following application of this principle. 

\begin{corollary}
If $\Gamma$ is a lattice in one of the Lie groups $\SO(n,1)$, $\SU(n,1)$ or $\Sp(n,1)$ then $\Gamma$ is virtually diffuse. 
\label{lattice_diff}
\end{corollary}

\begin{proof}
A lattice $\Gamma$ in a rank one Lie group $G$ is a geometrically finite group (cf.\ 5.4.2 in \cite{Bowditch_geom_fin}),
hence we need to prove that the parabolic isometries contained in $\Gamma$ have only roots of unity as eigenvalues. In the case that $\Gamma$ is arithmetic there is a quick argument: for any $\gamma\in\Gamma$, the eigenvalues of $\gamma$ are algebraic numbers. If in addition $\gamma$ is parabolic, then all its eigenvalues are of absolute value one as well as their conjugates (because the group defining $\Gamma$ is compact at other infinite places). A theorem of Kronecker \cite[Theorem 1.31]{Everest_Ward} shows that any algebraic integer in $\CC$ whose Galois conjugates are all of absolute value one must be a root of unity, and it follows that the eigenvalues of $\gamma$ are roots of unity. 

One can also use a more direct geometric argument to prove this in full generality. Let $P=MAN$ be a parabolic subgroup of $G$ 
which contains a parabolic element of $\Gamma$; then it is well-known that $\Gamma\cap P$ is contained in $MN$
(see the proof of Lemma \ref{nofun} above). Also $\Lambda = \Gamma\cap N$ is a lattice in $N$, in particular $\Lambda\bs N$ is compact
(this follows from the Margulis Lemma \cite[Proposition 3.5.1]{Bowditch_geom_fin},
which implies that horosphere quotients inject into $\Gamma\bs X$, and the finiteness
of the volume of $\Gamma\bs X$). Corollary \ref{lattice_diff} will then follow from the next lemma. 

\begin{lemma}
Let $N$ be a simply connected nilpotent Lie group containing a lattice $\Lambda$, 
and $Q\le{\rm Aut}(N)$ a subgroup which preserves $\Lambda$, 
all of whose elements have only eigenvalues of absolute value one 
(in the representation on the Lie algebra $\mathfrak{n}$). Then these eigenvalues are in fact roots of unity.
\end{lemma}

\begin{proof}
The exponential map $\exp: \mathfrak{n}\to N$ is a diffeomorphism. 
By \cite[Theorem 2.12]{Raghunathan1972}, there is a lattice $L$ in the vector space $\mathfrak{n}$ such that $\langle\exp(L)\rangle = \Lambda$.
It follows that the adjoint action of $Q$ preserves $L$, hence for any $q\in Q$ the characteristic polynomial of $\ad(q)$ has integer coefficients,
hence its eigenvalues are the conjugates of some finite set of algebraic integers. Since they are also all of absolute value one it follows from Kronecker's theorem that they must be roots of unity.
\end{proof}

It follows that, in the above setting, the image of $\Gamma\cap P$ in $M$ has a finite-order image in ${\rm Aut}(N)$ where $M$ acts by conjugation. This action is faithful (because an element of $M$ cannot act trivially of an horosphere associated to $N$, otherwise it would act trivially on the whole of $X$ since it preserves these horospheres) and it follows that the hypothesis on eigenvalues in Theorem \ref{rank1} is satisfied by $\Gamma$. 
\end{proof}

%%%%%%%%%%%%%%%%%%%%%%%%%%%%%%

\subsubsection{Action on the boundary}

\begin{theorem}
Let $\Gamma,G$ be as in the statement of Theorem \ref{rank1}.
Then there is a finite-index $\Gamma'\subset\Gamma$ such 
that for any parabolic fixed point $\xi\in\pl X$ for $\Gamma'$ with stabilizer $\Lambda_\xi$ in $\Gamma'$ the action of $\Gamma'$ on $\Gamma'/\Lambda_\xi$ is diffuse.
\label{real_bd_diff}
\end{theorem}

\begin{proof}
We take a finite-index subgroup $\Gamma'\le\Gamma$ as in the proof of Theorem \ref{rank1} above. The key point is the following lemma. 

\begin{lemma}
There is a dense subset $S_{\Gamma'}\subset X$ such that for any $x_0\in S_{\Gamma'}$ and any parabolic fixed point $\xi$ of $\Gamma'$, if $b_\xi$ is a Busemann function at $\xi$ we have 
\begin{equation}
 \forall g\in\Gamma',\, g\not\in\Lambda_{\xi} : \max\left( b_\xi(g x_0), b_\xi(g^{-1} x_0) \right) > b_\xi(x_0). 
\label{busemann_move}
\end{equation}
\end{lemma}

\begin{proof}
Fix $\xi$ and $b_\xi$ as in the statement. By definition of a Busemann function there is a unit speed geodesic ray $\sigma: [0,\infty[ \to X$ running to $\xi$ in $X\cup\pl X$, such that for all $x\in X$ we have
$$
b_\xi(x) = \lim_{t\to+\infty} \left( d(x,\sigma(t)) - t \right).
$$
On the other hand, by construction of $\Gamma'$ (using Lemma \ref{minsep}) we know that for all axial isometries $g \in\Gamma'\setminus\{1\}$ we have 
$$
\forall t \geq 0\: \max\left(d(gx_0,\sigma(t)),d(g^{-1}x_0, \sigma(t))\right) \ge d(x_0, \sigma(t)) + 1 ; 
$$
passing to the limit we obtain \eqref{busemann_move} for all such $g$ and for any choice of $x_0$. 

Now we show that for certain generic $x_0$ the same is true for unipotent isometries. 
In any case, for any unipotent isometry $g$ of $X$, it follows from Proposition \ref{parsep} and the same argument as above that
\begin{equation}
\max(b_{\xi} (g^{-1} x_0), b_{\xi} (gx_0)) \ge b_{\xi}(x_0)
\label{unip_Bus0}
\end{equation}
for all $x_0$. We want to choose $x_0$ in order to be able to rule out equality if $g\in\Gamma'-\Lambda_\xi$.
For a given unipotent isometry $\eta$ and a $\zeta\in \pl X$ with $\eta\zeta\not= \zeta$ define
$$
E_{\zeta,\eta} = \{ x\in X \:|\: b_\zeta(\eta x) = b_\zeta(x) \}
$$
(note that this does not depend on the choice of the Busemann function $b_\zeta$). This is an embedded hyperplane in $X$, and hence (by Baire's theorem) the subset
$$
S_{\Gamma'} = X-\bigcup_{\zeta,\eta} E_{\zeta,\eta}
$$
where the union runs over all parabolic elements $\eta$ of $\Gamma'$ and all parabolic fixed points $\zeta$ of $\Gamma'$, is dense in $X$.
Moreover, by the definition of $S_{\Gamma'}$, for $x_0\in S_{\Gamma'}$ we never have $b_\xi(g x_0) = b_\xi(x_0)$ for
any unipotent $g \in\Gamma'$ with $g \xi\not=\xi$.
Thus \eqref{unip_Bus0} has to be a strict inequality.
\end{proof}

Let $\xi_0\in\pl X$ be a parabolic fixed point of $\Gamma'$ and $b_{\xi_0}$ a Busemann function at $\xi_0$. We write $\Lambda = \Lambda_{\xi_0}$.
The function $b_{\xi_0}$ is $\Lambda$-invariant and if we choose some $x_0\in X$ we may define a function $f = f_{x_0}$ on $\Gamma'/\Lambda$ by 
\begin{equation}
f(\gamma\Lambda) = b_{\xi_0}(\gamma^{-1}x_0) = b_{\gamma\xi_0}(x_0). 
\label{def_f}
\end{equation}
By the lemma this function satisfies
\begin{equation}
\forall \gamma\Lambda\in\Gamma'/\Lambda,\, \forall g\in\Gamma',\, g\not\in\gamma\Lambda\gamma^{-1} : \max\left( f(g\gamma\Lambda), f(g^{-1}\gamma\Lambda) \right) > f(\gamma\Lambda),
\label{htot}
\end{equation}
whenever $x_0 \in S_{\Gamma'}$.
Indeed, we have
$$
\max\left( f(g\gamma\Lambda), f(g^{-1}\gamma\Lambda) \right) = \max\left(b_{\gamma\xi_0}( gx_0), b_{\gamma\xi_0}( g^{-1}x_0)\right)
$$
and according to \eqref{busemann_move} the right-hand side is strictly larger than $b_{\gamma\xi_0}(x_0)=f(\gamma\Lambda)$.

The existence of a function $f$ satisfying \eqref{htot} implies that the action $\Gamma'$ on $\Gamma'/\Lambda$ is weakly diffuse, 
i.e.~every non-empty finite subset $A\subset\Gamma'/\Lambda$ has at least {\it one} extremal point. Indeed, any $a\in A$ such that $f(a)$ realizes the maximum of $f$ on $A$ is extremal in $A$. 

Using an additional trick we can actually deduce diffuseness.
Let $A\subset\Gamma'/\Lambda$ be finite with $|A|\ge 2$, and let $a$ be an extremal point. 
By shifting $A$ we can assume that $a=\Lambda$. Now let $\xi_0$ be the fixed point of $\Lambda$ and $b_{\xi_0}$ a Busemann function.
Choose
$x_0\in S_{\Gamma'}$ such that $x_0$ is (up to $\Lambda$) the only point realizing the minimum of $b_{\xi_0}$ on $\Gamma' x_0$
(this is possible by taking $x_0$ in a sufficiently small horoball at $\xi_0$, since $S_{\Gamma'}$ is dense) and define $f$ on $\Gamma'/\Lambda$ as in \eqref{def_f}. 
By construction $f$ takes it's minimal value at $a$. So let $b \in A$ be a point where $f$ takes a maximal value. By the given argument $b$ is extremal in $A$. 
On the other hand $f(b) > f(a)$ and so $b \neq a$. We conclude that $A$ has at least two extremal points.
\end{proof}

%%%%%%%%%%%%%%%%%%%%%%%%%%%%%%%%%%%%%%%%%%%%%%%%%%%%%%%%%%%%%%%%%%%%%%%%%%%%%%%%

\section{Fundamental groups of three--manifolds}

\label{3_mfd}

In this section we prove Theorem \ref{thm_3manifolds}, whose statement we recall now :

\begin{theorem*}
Let $M$ be a compact three--manifold and $\Gamma = \pi_1(M)$ its fundamental group. Then there is a finite-index subgroup $\Gamma'\le\Gamma$ which is diffuse. 
\end{theorem*}

The proof is a rather typical application of Geometrization. We begin with an algebraic result on graph products, afterwards
we use it to construct a suitable covering (cf.~\cite{Hempel}).

%%%%%%%%%%%%%%%%%%%%%%%%%%%%%%%%%%%%%%%%%%%%%%%%%%%%%%%%%%%%

\subsection{Algebraic preliminaries: a gluing lemma}

Bowditch \cite{Bowditch_diff} showed that if $\Gamma$ is the fundamental group of a graph of groups such that
for any vertex group $\Gamma_i$ and adjacent edge group $\Lambda_i$, 
both the group $\Lambda_i$ and the action of $\Gamma_i$ on $\Gamma_i/\Lambda_i$ are diffuse, then $\Gamma$ is diffuse. 
In order to glue manifolds it is necessary to understand graph products of virtually diffuse groups.
For free products there is a very simple argument.

\begin{lemma}
 The free product $G= G_1 * G_2$ of two virtually diffuse groups $G_1$ and $G_2$ is again virtually diffuse.
\label{free_prod}
\end{lemma}
\begin{proof}
    Let $H_i \leq_f G_i$ be a finite index diffuse subgroup.
    Consider the homomorphism $\phi: G \to G_1 \times G_2$. 
    The kernel $K$ of $\phi$ is a free group (cf.~I. Prop.~4 in \cite{Serre-Trees}).
    Let $H$ denote the inverse image of $H_1 \times H_2$ under $\phi$.
    The subgroup $H$ has finite index in $G$ and $H \cap K$ is a free group.
    We get a short exact sequence
    \begin{equation*}
        1 \longrightarrow K\cap H  \longrightarrow H \to H_1\times H_2 \longrightarrow 1.
    \end{equation*}
     From Theorem 1.2 in \cite{Bowditch_diff} we see that $H$ is diffuse.
\end{proof}
Note that the same argument shows that the free product of diffuse groups is diffuse.
In order to understand amalgamated products and HNN extensions of virtually diffuse groups one needs to 
argue more carefully.

We will use the Bass-Serre theory of graph products of groups. We shall use the notation of \cite{Serre-Trees}.
Recall that a graph of groups $(G,Y)$ is a \emph{finite} graph $Y$ with vertices $V(Y)$ and edges $E(Y)$.
Every edge $e$ has an origin $o(e)\in V(Y)$ and a terminus $t(e)\in V(Y)$.  Moreover for every edge there is an opposite edge~$\bar{e}$.
To every vertex $P \in V(Y)$ and every edge $e \in E(Y)$ there are attached groups $G_P$ and $G_e=G_{\bar{e}}$. Moreover, for every edge $e$ there is 
a monomorphism $i_e: G_e \to G_{t(e)}$ usually denoted by $a \mapsto a^e$.
To a graph of groups one attaches a fundamental group $\pi_1(G,Y)$ -- the graph product.

Let $(G,Y)$ be a graph of groups. A \emph{normal subcollection} $(N,Y)$ consists of two families $(N_P\normal G_P)_{P\in V(Y)}$ and $(N_e\normal G_e)_{e\in E(Y)}$ of normal subgroups in the
vertex and edge groups
which are \emph{compatible} in the sense that
\begin{equation*}
     i_e(N_e) = i_e(G_e) \cap N_{t(e)} \quad \text{ and } \quad N_e = N_{\bar{e}}
\end{equation*}
 for every edge $e\in E(Y)$.
We say that $(N,Y)$ is of finite index, if for every vertex $P$ the index of $N_P$ in $G_P$ is finite.

\begin{lemma}\label{lem:virtuallyFree}
 Let $(G,Y)$ be a graph of finite groups. The fundamental group $\Gamma = \pi_1(G,Y)$ is residually finite and virtually free.
\end{lemma}
\begin{proof}
 The residual finiteness follows from Theorem 3.1 of Hempel \cite{Hempel}. 
 To apply his result we need to specify sufficiently small normal subcollections $(H,Y)$ in $(G,Y)$ such for every $P\in V(Y)$ the group $H_P$ has finite index in $G_P$.
 Since we are dealing with finite groups it is easy to check that we can simply choose $H_P =\{1\}$ and $H_e=\{1\}$ for every vertex $P$ and edge $e$.

 Using that $\Gamma$ is residually finite, we can find a finite index subgroup $N \normal \Gamma$ which intersects the embedded vertex group $G_P$ trivially for any of the finitely many vertices
 $P\in V(Y)$. Therefore, the subgroup $N$ acts freely (without edge inversion) on the Bass-Serre tree associated with the graph $(G,Y)$. We deduce that $N$ is a free group \cite[I.~Thm.~4]{Serre-Trees}.
\end{proof}

Let $(G,Y)$ be a graph of groups and let $(N,Y)$ be a normal subcollection. To such a data we can associate a quotient graph of groups $(H,Y)$ where
$H_P = G_P/N_P$ and $H_e = G_e/N_e$ for all verticies $P$ and edges $e$.
There is a natural surjective quotient morphism $q: \pi_1(G,Y) \to \pi_1(H,Y)$. 
We are now able to state and prove the (algebraic) gluing Lemma.

\begin{lemma}[Gluing Lemma]
 Let $(G,Y)$ be a graph of groups such that 
   \begin{enumerate}[(i)]
    \item\label{assump:edgeDiff}   every edge group $G_e$ is diffuse
    \item\label{assump:virtDiff}  there is a normal subcollection $(N,Y)$ of finite index such that for every edge
           $e \in E(Y)$ the group $N_{t(e)}$ acts diffusely on $G_{t(e)}/i_e(G_e)$.
   \end{enumerate}
In this case the fundamental group $\Gamma = \pi_1(G,Y)$ is virtually diffuse.
\label{glue}
\end{lemma}
\begin{proof}
   Consider the associated quotient morphism $q: \Gamma \to \pi_1(H,Y)$.
   The kernel $\frN$ of $q$ is the normal subgroup generated by the groups $(N_P)_{P\in V(Y)}$.
   Let $\Gamma$ and $\frN$ act on the Bass-Serre tree $T$ associated with $(G,Y)$. The stabilizer in $\Gamma$ (resp.~$\frN$) of a vertex $v \in V(T)$ above $P \in V(Y)$ is isomorphic to $G_P$ (resp.~$N_P$).
   It acts on the set of adjacent edges $E(v) \subset E(T)$. 
   Pick an edge $e \in E(Y)$ with $t(e)=P$. As a set with $G_P$ action $E(v)$ is isomorphic to $G_P/i_e(G_e)$.
   By assumption \eqref{assump:virtDiff} the action of $N_P$ on $G_P/i_e(G_e)$ is diffuse. By a result of Bowditch \cite[Prop.~2.2]{Bowditch_diff} we deduce that $E(T)$ is a diffuse $\frN$
   set. Since the edge groups are assumed to be diffuse, we see that $\frN$ is diffuse.

   The quotient $(H,Y)$ is a graph of finite groups, we know from Lemma \ref{lem:virtuallyFree} that it is virtually free.
   Since free groups are diffuse, the short exact sequence
    \begin{equation*}
        1 \longrightarrow \frN \longrightarrow \Gamma \longrightarrow \pi_1(H,Y) \longrightarrow 1
    \end{equation*}
    implies the assertion by Thm.~1.2 (2) of \cite{Bowditch_diff}.
\end{proof}

%%%%%%%%%%%%%%%%%%%%%%%%%%%%%%%%%%%%%%%%%%%%%%%%%%%%%%%%%%%%

\subsection{Geometrization and the proof of Theorem \ref{thm_3manifolds}}

\subsubsection{Definitions}
\label{geom_conj}

We recall here the definitions which allow to state the Geometrization Theorem
which was conjectured by W.~Thurston (\cite{Thurston_geom}, see also \cite{Scott}) and proven by
G.~Perelman \cite{Perelman1, Perelman2} (see also \cite{KL_P} for a complete account of Perelman's proof).

In the following we consider (without loss of generality) only orientable manifolds.
A three--manifold $M$ is called \emph{irreducible} if all embedded 2-spheres in $M$ bound a ball.
A manifold is \emph{prime} if it is irreducible or homeomorphic to $S^1 \times S^2$.  
According to the Kneser--Milnor decomposition every closed three--manifold 
is a finite connected sum of prime manifolds.
A closed irreducible manifold $M$ has a further topological decomposition, called the Jaco--Shalen--Johansson
decomposition, which consists in a canonical collection of embedded, essentially disjoint 2-tori
in $M$ (see \cite{Jaco_book}).
The Geometrization Theorem states that every connected component of the complement in $M$
of this collection of tori is either a finite volume hyperbolic manifold or Seifert fibered.

%%%%%%%%%%%%%%%%%%%%%%%%%%%%%%

\subsubsection{Virtual diffuseness}
The following lemma treats the pieces of the Geometrization Theorem. It is the key ingredient for Theorem \ref{res_dim3}.
\begin{lemma}
Let $M$ be a compact three--manifold with incompressible toric boundary. 
If $M$ is either hyperbolic of finite volume or Seifert fibered, then $\Gamma=\pi_1(M)$ contains a diffuse subgroup $\Gamma'$ of finite index.
Moreover, if $M$ has non-empty boundary, then for almost all primes $p$ the group $\Gamma'$ can be chosen so that
for any peripheral subgroup $\Lambda$ of $\Gamma$ 
 \begin{enumerate}[ (a)]
  \item the $\Gamma'$-action on $\Gamma/\Lambda$ is diffuse and
  \item  $\Gamma'\cap\Lambda$ is the characteristic subgroup of index $p^2$.
 \end{enumerate}
\label{geom_vd}
\end{lemma}

\begin{proof}
Assume first that $M$ is closed. If $M$ is Seifert fibered, then $\pi_1(M)$ is a an extension of a group which is virtually a surface group by a cyclic group $C$ (cf.\ Lemma 3.2 in \cite{Scott}).
If $C$ is infinite, such a group is virtually diffuse by the results of Bowditch \cite{Bowditch_diff}.
Otherwise $M$ is covered by $S^3$ and the fundamental group is finite.
If $M$ is hyperbolic, then the virtual diffuseness follows from Theorem \ref{rank1_lattice:intro}.

Now we turn to the case where $M$ has non-empty boundary.
Assume first that $M$ is hyperbolic. In $\pi_1(M)$ there are only finitely many hyperbolic elements $h_1, \dots, h_m$
with translation length less than $2\log(1+\sqrt{2})$ (cf.\ Lemma~\ref{fin_geod}). 
By Lemma 4.1 of \cite{Hempel} we can find, for almost all primes $p$, a normal subgroup of finite index $\Gamma'_p \leq \pi_1(M)$ which does not contain $h_1,\dots,h_m$ and
which intersects the peripheral subgroup in the characteristic subgroup of index~$p^2$.
Using Theorem \ref{real_bd_diff} such a group $\Gamma'_p$ is diffuse and acts diffusely on $\Gamma/\Lambda$ for any peripheral $\Lambda$.

Finally assume that $M$ is Seifert fibered.
There is a short exact sequence
\begin{equation*}
  1 \longrightarrow \ZZ \longrightarrow \pi_1(M) \stackrel{q}{\longrightarrow} G \longrightarrow 1
\end{equation*}
where $\ZZ$ is generated by the regular fibers and $G$ is the fundamental group of a two dimensional orbifold $B$ with non-empty boundary.
Taking the finite index subgroup of elements commuting with the regular fibres (which contains the peripheral subgroups), we can assume that the extension is central.
Since the boundary of $M$ is incompressible, the simple closed boundary curves $d_1, \dots d_k$ of $B$ have infinite order in $G$.
For almost all primes $p$ there is a free normal subgroup $G_p \subset G$ of finite index such that $G_p \cap \langle d_i \rangle = \langle d_i^p \rangle$.
One way to see this is to argue using the presentation of $G$ as given in \cite[12.1]{Hempel_Book}.
Geometrically this can be seen as follows: Glue a disc with a $p$-cone point into every boundary curve of $B$. For almost all $p$ the resulting orbifold $B_p$ is good and has hence
a finite sheeted regular cover $\widetilde{B_p}$ which is a manifold. 
Removing the inverse images of the glued discs we obtain a finite covering space $S_p$ of $B$ which is a compact surface so that 
the boundary components are $p$-fold covers of the boundary components of $B$. Since a compact surface with non-empty boundary has a free fundamental group the claim follows.

The finite sheeted cover $\widetilde{M}_p$ corresponding to $q^{-1}(G_p)$ has fundamental group isomorphic to $\ZZ \times G_p$. Finally the group $\Gamma'_p = p\ZZ \times G_p$ is diffuse and
intersects the peripheral subgroups in their characteristic subgroups of order $p^2$.
It remains to verify that 
that the action of $\Gamma'_p$ on $\Gamma/\Lambda$ is diffuse.
This action factors through the group $G_p$ and so the assertion follows
from Theorem \ref{real_bd_diff} if we embed $G_p$ as a discrete subgroup into $\SL_2(\RR)$.

There is another argument for the diffuseness of this action: We can assume that the surface $S_p$
has more than one boundary component, and so the boundary curves can be choosen to be part of a free generating set.
Let $F$ be a free group and $f \in F$ an element of a free generating set, then by Prop.~2.2 in \cite{Bowditch_diff} 
the action of $F$ on $F/\langle f \rangle$ is diffuse.
\end{proof}

\subsubsection{Proof of Theorem \ref{thm_3manifolds}}
Let $M$ be a compact three--manifold; by doubling it (and since virtual diffuseness passes to subgroups)
we may assume that it is in fact closed.
By Lemma \ref{free_prod} and the Kneser--Milnor decomposition we may assume that $M$ is irreducible.
An irreducible manifold admits a geometric decomposition (see \ref{geom_conj}), which yields a decomposition of $\pi_1(M)$
as a graph of groups whose vertex groups are fundamental groups of Seifert fibered or hyperbolic manifolds and the edge groups are peripheral subgroups.
Choosing a prime number $p$ which is admissible for all the occuring pieces, it follows from Lemma \ref{geom_vd} that
this graph of groups has a normal subcollection which satisfies the hypotheses of Lemma \ref{glue}. 

%%%%%%%%%%%%%%%%%%%%%%%%%%%%%%%%%%%%%%%%%%%%%%%%%%%%%%%%%%%%

\subsection{Three--dimensional infra-solvmanifolds}
\label{3d_infrasol}

A three--dimensional {\it solvmanifold} is a (left) quotient of the solvable Lie group 
$$
\sol = \RR^2 \rtimes \RR ;\quad t\cdot x = \begin{pmatrix} e^t & \\ & e^{-t}\end{pmatrix}\cdot x
$$ 
by a discrete subgroup; an {\it infra-solvmanifold} is a quotient of such by a finite group acting freely.
Any left-invariant Riemannian metric on $\sol$ induces a complete Riemannian metric on an infra-solvmanifold.
A compact solvmanifold is finitely covered
by a torus bundle (see for example \cite[Theorem 5.3 (i)]{Scott}),
hence its fundamental group contains a subgroup of finite index which is an extension of $\ZZ^2$ by $\ZZ$.
More precisely, this group will be isomorphic to some
$$
\Gamma_A = \langle \ZZ^2,t \:|\: \forall v\in\ZZ^2, tvt^{-1} = Av \rangle
$$  
where $A\in\SL_2(\ZZ)$ is not unipotent. Such a group is diffuse by \cite[Thm 1.2]{Bowditch_diff}.
On the other hand we will now explain how to construct infra-solvmanifolds (so-called `torus semi-bundles')
of dimension three with zero first Betti number (by gluing I-bundles over Klein bottles, see \cite{Hatcher_3D}),
which are then not locally indicable and hence not diffuse. 

The following result is a special case of Proposition \ref{prop_infrasolcom}. We shall give another geometric argument
(see also \cite[Corollary 8.3]{HaLee} for a complete description of the groups of isometries acting properly discontinously,
freely and cocompactly on $\sol$ from which it follows easily). 

\begin{prop}
In every commensurability class of compact three--dimensional infra-solvmanifolds there is a manifold with non-diffuse fundamental group.
\label{Sol}
\end{prop}

\begin{proof}

We give a topological construction (see \cite{Hatcher_3D}):
let $N$ be the non-trivial I-bundle over the Klein bottle,
so that $\pl N=\TT^2$. Then for any mapping class $B\in\SL_2(\ZZ)$ of $\TT^2$ the gluing $M=N\cup_\phi N$ has $b_1(M)=0$ or is Seifert; in the former case it is a sol-manifold and is doubly covered by the torus bundle with holonomy $A=SB^{-1}SB$ where $S$ is the symmetry $(x,y)\mapsto (-x,y)$. In this way we get all $A$s of the form
$$
A = \begin{pmatrix} b & 2a \\ 2c & b \end{pmatrix}
$$
where $a,b,c\in\ZZ$: this follows from a direct computation. On the other hand it is easy to see that every hyperbolic conjugacy class in $\SL_2(\ZZ)$ contains a matrix with both diagonal coefficients equal (send a geodesic line orthogonal to its axis to $(0,\infty)$) and then we can take the square or cube to ensure that the off-diagonal coefficients are even. 
\end{proof}

%%%%%%%%%%%%%%%%%%%%%%%%%%%%%%%%%%%%%%%%%%%%%%%%%%%%%%%%%%%%%%%%%%%%%%%%%%%%%%%%

%%%%%%%%%%%%%%%%%%%%%%%%%%%%%%%%%%%%%%%%%%%%%%%%%%%%%%%%%%%%%%%%%%%%%%%%%%%%%%%%

\section{Computational aspects}

\subsection{Finding ravels}
\label{computation_ravel}
Given a group $\Gamma$ it is a substantial problem to decide whether or not the group is diffuse. 
To a certain degree this problem is vulnerable to a computational approach which will be explained in this section.

For all the following algorithms we suppose that we have a way of solving the word problem in a given group $\Gamma$;
in practice we used computations with matrices to do this. We will not make reference to the group $\Gamma$ in the algorithms. 

The first algorithm determines, given a finite subset $A$ of $\Gamma$ and an element $a\in A$, whether $a$ is extremal in $A$ or not.

\begin{algorithm}[H]
\caption{Given $a\in A \subset \Gamma$, determines if $a$ is extremal in $A$}
\begin{algorithmic}[1] 
\Function{IsExtremal}{$a,A$}
  \State $B = A\setminus\{a\}$
  \ForAll{$b\in B$}
    \If{$ab^{-1}a\in A$}
      \Return False 
                   \Comment{If $b=ga$ and $g^{-1}a = aba^{-1}\in A$ then $a$ is not extremal.}
    \EndIf
  \EndFor
  \State \Return True 
\EndFunction
\end{algorithmic}
\end{algorithm}

The following algorithm returns the largest ravel contained in $A$ by succesively removing extremal points. If $A$ contains no ravel, then it returns the empty set. Of course, the algorithm is not able to decide if a ravel exists at all (hence is of no use to prove that a group is diffuse). 

\begin{algorithm}[H]
\label{algo}
\caption{Given $A \subset \Gamma$, finds the largest ravel contained in $A$}
\begin{algorithmic}[1] 
\Function{FindRavel}{$A$}
    \ForAll{$a \in A$} 
      \If{IsExtremal($a,A$) = True}
        \Return FindRavel($A \setminus \{a\}$)
      \EndIf
    \EndFor
  \State \Return A  
                 \Comment{No extremal point was found in $A$, so $A$ is a ravel or empty}
\EndFunction
\end{algorithmic}
\end{algorithm}

Finally, it may be of interest to determine mininal ravels; the following algorithm, starting from a ravel $A$, finds a minimal one contained in $A$ (note that the result may depend on the order on which the elements of $A$ are looped over). 

\begin{algorithm}[H]
\caption{Given a ravel $A \subset \Gamma$, finds a minimal ravel contained in $A$}
\begin{algorithmic}[1]
\Function{MinRavel}{$A$}
  \ForAll{$a\in A$}
    \State B = FindRavel($A\setminus\{a\}$)
    \If{$B\not=\emptyset$}
      \Return{MinRavel($B$)}
    \EndIf
  \EndFor
\State \Return $A$
\EndFunction
\end{algorithmic}
\end{algorithm}

To prove Proposition \ref{Weeks_nd} we ran (with two different implementations  in Magma \cite{Magma}  and in Sage/Python \cite{sage}) the algorithms to test diffuseness on the group with presentation 
$$
\langle a,b| a^2b^2a^2b^{-1}ab^{-1},\, a^2b^2a^{-1}ba^{-1}b^2\rangle,
$$ 
which is the fundamental group of the Weeks manifold, the hyperbolic three--manifold of smallest volume. We actually used the representation to $\SL_2(\CC)$ given in the proof of Proposition 3.2 in \cite{CFJR}:
$$
a = \begin{pmatrix} x & 1 \\ 0 & x^{-1} \end{pmatrix},\quad b = \begin{pmatrix} x & 0 \\ 2 - (x + x^{-1}) & x^{-1} \end{pmatrix}
$$
where 
$$
x^6 + 2x^4 - x^3 + 2x^2 + 1 = 0. 
$$
It turns out that the word metric ball of radius four in the generators $a,b$ contains a ravel of cardinality 141
(further computation showed that the latter contains a minimal ravel of cardinality 23). 

\subsection{Implementation}
\subsubsection{SAGE}

The Sage implementation of the algorithm (for linear groups) can be found as the file \texttt{algo\char`_ diffuse.py} in \cite{ancillary}. 
It has to be run in a Sage environment, and the main function is \texttt{max\char`_ diff}, which takes as input a pair $(S,{\rm M})$ where $\rm M$ is a Sage MatrixSpace object, and $S$ a collection of invertible matrices in $\rm M$. Its output is the (possibly empty) maximal ravel contained in $S$. The file also contains the function \texttt{ball}, which inputs a triple $(r,{\rm gens},{\rm M})$ which computes the ball of radius $r$ in the group generated by the set ${\rm gens}$ of invertible matrices in ${\rm M}$ (in the word metric associated to ${\rm gens}$). The file \texttt{Weeks.sage} in \cite{ancillary} can be run directly in a Sage environment and outputs a ravel of cardinality 141 in the Weeks manifold group. 

\subsubsection{MAGMA}
An implementation for the MAGMA computer algebra system can be found as \texttt{diffuse\char`_ MAGMA} in \cite{ancillary}. 
It includes functions \texttt{findRavel}, \texttt{findMinRavel} and a procedure \texttt{BallWeeks} to generate a ball of given radius in the Weeks manifold group.
To compute a ravel in the Weeks manifold group run the following lines
\begin{verbatim}
   |> B := BallWeeks(4);
   |> findRavel(B);
\end{verbatim}

%%%%%%%%%%%%%%%%%%%%%%%%%%%%%%%%%%%%%%%%%%%%%%%%%%%%%%%%%%%%%%%%%%%%%%%%%%%%%%%%

\appendix

\section{A diffuse group which is not left-orderable\\ by Nathan M. Dunfield}
\label{appendix}
\markleft{{\small{\sc{Nathan Dunfield}}}}
\noindent
This appendix is devoted to the proof of

\begin{theorem}\label{thm:hypmain}
  Let $N$ be the closed orientable hyperbolic \3-manifold defined
  below.  Then $\pi_1(N)$ is diffuse but not left-orderable.
\end{theorem}
\noindent
This example was found by searching through the towers of finite
covers of hyperbolic \3-manifolds studied in \cite[\S
6]{CalegariDunfield2006}.  There, each manifold has $b_1 = 0$ (which
is necessary for $\pi_1$ to be non--left-orderable) and the length
of the systole goes to infinity (so that we can apply Bowditch's
criterion for diffuseness).  We begin by giving two 
descriptions of $N$, one purely arithmetic and the other purely
topological.

\subsection{Arithmetic description}  \label{sec:arithmetic}

Throughout this section, a good reference for arithmetic hyperbolic
\3-manifolds is \cite{MaclachlanReid2003}. Let $K = \Q(\a)$ be the
number field where $\a^3 + \a - 1 = 0$; this is the unique cubic field
with discriminant $-31$.  It has one real embedding and one pair of
complex embeddings; our convention is that the complex place
corresponds to $\a \approx -0.3411639 + 1.1615414 i$.  Its integer
ring $\OK$ has unique factorization, so we will not distinguish
between prime elements and prime ideals of $\OK$.  The unique prime of
norm 3 in $\OK$ is $\pi = \a + 1$, and let $D$ be the quaternion
algebra over $K$ ramified at exactly $\pi$ and the real place of $K$.
Concretely, we can take $D$ to be generated by $i$ and $j$ where
$i^2 = -1$, $j^2 = -3$ and $k = i j = -j i$.  The manifold $N$ will be
the congruence arithmetic hyperbolic \3-manifold associated to $D$ and
the level $\pi^3$, whose detailed construction we now give.

Let $\OD$ be a maximal order in $D$; this is unique up to conjugation
by \cite[Example 6.7.9(3)]{MaclachlanReid2003}.  Let $\OD^1$ denote
the elements of $\OD$ of (reduced) norm 1.  At the complex place of
$K$, the algebra $\C \otimes_K D$ is just the matrix algebra
$M_2(\C)$.  Let $\Lambda$ be the subgroup of
$\PSL_2(\C) \cong \Isom^+{\H^3}$ which is the image of $\OD^1$ under
the induced map $D^1 \to \SL_2(\C) \to \PSL_2(\C)$.  Since $D$ is a
division algebra, $\Lambda$ is a cocompact lattice.  Let $K_\pi$ be
the $\pi$-adic completion of $K$, which is isomorphic to $\Q_3$.  Let
$D_\pi = K_\pi \otimes_K D$, which is the unique quaternion division
algebra over $K_\pi$ \cite[\S 2.6]{MaclachlanReid2003}.  Define
$w \maps D_\pi \to \Z$ by $w = \nu \circ n$ where
$\nu \maps K_\pi \to \Z$ is the (logarithmic) valuation and
$n \maps D_\pi \to K_{\pi}$ is the norm function.  Then
$\Opi = \setdef{u \in D_\pi}{w(u) \geq 0}$ is the valuation ring of
$D_\pi$ and $\cQ = \setdef{u \in D_\pi}{w(u) \geq 1}$ is the maximal
two-sided ideal in $\Opi$ (compare \cite[\S 6.4]{MaclachlanReid2003}).
Define $\Gamma$ to be the image of $\OD^1 \cap \left(1 + \cQ^3\right)$
in $\PSL_2(\C)$, and let $N$ be the associated hyperbolic orbifold
$\hypquo{\Gamma}$. 

We claim that $\Gamma$ is torsion-free and hence $N$ is a
manifold.  First note that
$\cQ^n = \setdef{u \in D_\pi}{w(u) \geq n}$.  Now for
$\gamma \in \Gamma$, we have $\gamma = 1 + q$ for $q \in \cQ^3$; from
$n(\gamma) = 1$ we get that $\tr(\gamma) - 2 = -n(q)$ and thus
$\tr(\gamma) - 2 \in \pi^3$ since $w(q) \geq 3$.  If $\gamma$ has
finite order, then $\tr(\gamma) = \xi + \xi^{-1}$ where $\xi$ is a
root of unity.  Since $\tr(\gamma) \in \OK$, it would have to be one
of $\{-1, 0, 1\}$ and none of those are $2 \bmod
\pi^3$.  So $N$ is a manifold.

\subsection{Topological description} Let $M$ be the hyperbolic
\3-manifold $m007(3,2)$ from the Hodgson-Weeks census
\cite{HodgsonWeeks1994}; alternatively, $M$ is the $(-9/2, -3/2)$ Dehn
surgery on the Whitehead link $L$, where $+1$ surgery on $L$ yields
the figure-8 knot rather than the trefoil.  Then $\vol(M) \approx
1.58316666$ and $H_1(M; \Z) = \Z/3\Z \oplus \Z/9\Z$.  Let $N'$ be
the regular cover of $M$ corresponding to any epimorphism
$\pi_1(M) \to (\Z/3\Z)^2$; thus $\vol(N') \approx 14.24849994$.   We
will show:

\begin{prop}\label{prop:sameN}
  The hyperbolic manifolds $N$ and $N'$ are isometric.
\end{prop}

\begin{proof}
  We give a detailed outline, but many steps are best checked by
  rigorous computation; complete Sage \cite{sage} source code for this
  is available at \cite{ancillary}.  From a triangulation for the
  alternate topological description of $M$ as $m036(3, -1)$, Snap
  \cite{Snap, CoulsonGoodmanHodgsonNeumann2000} gives the group
  presentation
  \begin{equation}\label{eq:pi1M}
    \pi_1(M) = \spandef{a, b \vphantom{\big|}}{\rel{aaBaabbAbb} = 1, \ \rel{abbAbAAbAbb} = 1}
  \end{equation}
  where $A = a^{-1}$ and $B = b^{-1}$.  Moreover, Snap rigorously
  checks that $M$ is hyperbolic and that the holonomy representation
  $\pi_1(M) \to \PSL_2(\C)$ lifts to
  $\rho \maps \pi_1(M) \to \SL_2(\C)$ which is characterized (up to
  conjugacy) by
  $\tr\left(\rho(a)\right) = \tr\left(\rho(b) \right) = \a^2 + 1$ and
  $\tr\left(\rho(ab)\right) = \a$. 

  An $\OK$ basis for $\OD$ can be taken to be $\{1, i, x, y\}$ where
  $x = (i + j)/2$ and $y = (3\pi + 3\pi^2 i + \pi^2 j + \pi k)/6$.  If
  we define
  \begin{equation}\label{eq:quat}
  \abar = 1 + \a i + \a x +  (\a-1) y  \mtext{and} \bbar = -i \cdot
  \abar \cdot i
  \end{equation}
  then computing the norms and traces of
  $\big\{ \abar, \, \bbar, \, \abar \cdot \bbar \, \big\}$ and evaluating the relations
  in (\ref{eq:pi1M}) shows that $a \mapsto \abar$ and
  $b \mapsto \bbar$ gives a homomorphism
  $\pi_1(M) \to \OD^1 \leq \SL_2(\C)$ which is a conjugate of $\rho$.
  Henceforth, we identify $\pi_1(M)$ with the subgroup of $\OD^1$
  generated by $\big\{ \abar, \bbar \, \big\}$.

  Now, GAP or Magma \cite{GAP, Magma} easily checks that $\pi_1\left(N'\right)$
  is generated by
  \[
  \left\{\defgen{c}{a^3}, \defgen{d}{b^3}, \defgen{e}{baBA},
  \defgen{f}{bABa}\right\}
  \]
  with defining relators:
  \begin{align*}
    &\rel{DefDeceFdFcFe}  &\rel{DeceDecDCEfCEfCfDf} \\ 
    &\rel{ECEdFcDfDeceDeccFec} &\rel{fCfDecdcFecfDeceDec}
  \end{align*}
  To see that $\pi_1\left(N'\right) \leq \Gamma$, one just checks that
  $w(g - 1) = 3$ for $g$ in $\{c, d, e, f\}$ to confirm that each is
  in $1 + \cQ^3$.  By the volume formula
  \cite[Thm~11.1.3]{MaclachlanReid2003},
  $\vol(\hypquo{\Lambda}) \approx 0.26386111$ and hence
  $\left[\Lambda : \pi_1\left(N'\right)\right] = 54$.  On the other
  hand, one can calculate $[\Lambda : \Gamma]$ exactly as in the proof
  of Theorem~1.4 of \cite{CalegariDunfield2006}; while the number
  field in that example is $\Q(\sqrt{-2})$, in both examples
  $K_\pi \cong \Q_3$ and hence have isomorphic $D_\pi$.  Since it
  turns out that $[\Lambda : \Gamma]$ is also $54$, we have
  $\pi_1\left(N'\right) = \Gamma$ as claimed.
\end{proof}

Theorem~\ref{thm:hypmain} follows immediately from the following two
lemmas, whose proofs are independent of one another. 
\begin{lemma}\label{lem:systole}
  Let $N$ be the closed hyperbolic \3-manifold defined above.  Then
  $N$ has systole $\approx 1.80203613 > 2 \log\left(1 + \sqrt 2\right)$.  In
  particular, $\pi_1(N)$ is diffuse. 
\end{lemma}

\begin{lemma}\label{lem:nonorder}
  Let $N$ be the closed hyperbolic \3-manifold defined above.  Then
  $\pi_1(N)$ is not left-orderable. 
\end{lemma}

\begin{proof}[Proof of Lemma~\ref{lem:systole}]
  We will show that the shortest geodesics in $N$ correspond to
  elements $\gamma \in \Gamma = \pi_1(N)$ with $\tr(\gamma) = \a^2
  - \a$; one such element is $ec = \mathit{baBaa}$.  Since the
  translation length of $\gamma$ is given by
  \begin{equation}
    \min(\gamma) =  T\left(\tr(\gamma)\right) \mtext{where $T(z) =\real\big( 2 \arcosh(z/2)\big)$}
  \end{equation}
  the systole will thus have length $\approx 1.8020361$.  The
  conclusion that $\Gamma$ is diffuse follows immediately from
  Bowditch's criterion (iv) quoted above in Section 2.1.

  We will use the Minkowski geometry of numbers picture (see
  e.g.~\cite[\S I.5]{Neukirch1999}) to determine the possible traces
  of elements of $\Gamma$ with short translation lengths.  Let
  $\tauC \maps K \to \C$ be the preferred complex embedding and
  $\tauR \maps K \to \R$ be the real embedding.  We have the usual
  embedding from $K$ into the Minkowski space $K_\R = \R \times \C$
  given by $\iota = \tauR \times \tauC$, and the key fact is that
  $\iota( \OK )$ is a lattice in $K_\R$.   Thus the following set is finite:
  \[
  \cT = \setdef{t \in \OK \vphantom{\big|} }{
    \mtext{$\abs{\tauR(t)}\leq 2$, 
      $\abs{\tauC(t)} \leq 4$, and $t - 2 \equiv 0 \bmod \pi^3$}}
  \]
  We next show that $\cT$ contains $\tr(\gamma)$ for any
  $\gamma \in \Gamma$ with $\min(\gamma) \leq 2.5$.  That
  $\abs{\tauR(\tr \gamma)} \leq 2$ follows since $\Gamma$ is
  arithmetic: the quaternion algebra $D$ ramifies at the real place
  and so $D^1$ becomes $\mathrm{SU}_2$ there.  To see that
  $\min(\gamma) \leq 2.5$ implies $\abs{\tauC(t)} \leq 4$, note that
  $T(z)$ is minimized for fixed $\abs{z}$ on the real axis and that
  $T(4) < 2.6339$.

  To complete the proof of the lemma, we will show that
  $\cT = \{ 2, \a^2 - \a, \ -2\a^2 + \a - 1 \}$, which suffices since
  $T(-2\a^2 + \a - 1) \approx 2.33248166$.  The natural inner product
  on $K_\R$ is such that
  $\abs{ \iota(k)}^2 = \abs{\tauR(k)}^2 + 2 \abs{\tauC(k)}^2$ for all
  $k \in K$.  Hence any element of $\cT$ has norm $\leq 6$, and our
  strategy is to enumerate all elements of $\iota(\OK)$ to that norm
  and check which are in $\cT$.  A $\Z$-basis for $\OK$ is
  $\{1, \a, \a^2\}$, and the Gram matrix in that basis for the inner
  product on $K_\R$ has smallest eigenvalue $\approx 1.534033$.
  Regarding $\Z^3$ as having the standard norm from $\R^3$, this says
  that the natural map $\Z^3 \to \iota(\OK)$ is distance
  nondecreasing.  Hence every element of $\cT$ has the form $c_0 +
  c_1 \a + c_2 a^2$ where $c_i \in \Z$ with $c_1^2 + c_2^2 + c_3^2
  \leq 36$.  Computing $\cT$ is now a simple enumeration of the 
  925 such triples $(c_1, c_2, c_3)$.  See \cite{ancillary} for a short program
  which does this.  
\end{proof}

Turning to the proof of Lemma~\ref{lem:nonorder}, you will quickly see
that it was discovered by computer, using the method of \cite[\S
8]{CaDu}.  Verifying its correctness is a matter of
checking that 23 different elements in $\Gamma$ are the identity,
which can be easily done using the explicit quaternions given in
(\ref{eq:quat}); sample code for this is provided with \cite{ancillary}.

\begin{proof}[Proof of Lemma~\ref{lem:nonorder}]

  Assume $\Gamma$ is left-orderable and consider the positive cone
  $P = \setdef{\gamma \in \Gamma}{\gamma > 1}$.  We define some
  additional elements of $\Gamma$ by 
  \[
  \defgen{g}{aBABB} \quad
  \defgen{h}{abbAb} \quad
  \defgen{n}{aBBAB} \quad
  \defgen{m}{aBaab} \quad
  \defgen{v}{ABAAb}
  \]
  By symmetry, we can assume $g \in P$. We now try all the
  possibilities for whether the elements $\{g,h,n,d,c,m,v\}$ are in
  $P$ or not, in each case leading to the contradiction that $1 \in P$.
    
\pfassume{1}{h}
\pfassume{2}{n}
\pfassume{3}{d}
\pfassume{4}{c}
\pfassumefinal{5}{m}
\pfcontradiction{6}{\rel{cgndhmgmcmdhmchm}}
\pfassumefinal{5}{M}
\pfcontradiction{6}{\rel{MgndhdMgndhdMgnMhdMndMdMgndhdMgnMhdMgn}}
\pfassumefinal{4}{C}
\pfcontradiction{5}{\rel{hCggnhCgChCgd}}
\pfassume{3}{D}
\pfassume{4}{c}
\pfassumefinal{5}{m}
\pfcontradiction{6}{\rel{mcDnDmcDnDmcmnDmhDmDmcDnDmcmnDmc}}
\pfassumefinal{5}{M}
\pfcontradiction{6}{\rel{gnMDnMgnnMgncDnMg}}
\pfassumefinal{4}{C}
\pfcontradiction{5}{\rel{ChDnhCgDnhCggnnhCggnhCg}}
\pfassume{2}{N}
\pfassumefinal{3}{d}
\pfcontradiction{4}{\rel{NdhNgdNdhhNgdNdhdNdhNgdNgdNdhdNdhNgdN}}
\pfassume{3}{D}
\pfassume{4}{c}
\pfassumefinal{5}{m}
\pfcontradiction{6}{\rel{DmchNgmgmDNm}}
\pfassume{5}{M}
\pfassumefinal{6}{f}
\pfcontradiction{7}{\rel{hNcNfhMhNcNfhNhNfDf}}
\pfassumefinal{6}{F}
\pfcontradiction{7}{\rel{NhNcFMhNhNcFMhNcMhNcFcF}}
\pfassume{4}{C}
\pfassume{5}{v}
\pfassumefinal{6}{f}
\pfcontradiction{7}{\rel{ChCgChNfhf}}
\pfassumefinal{6}{F}
\pfcontradiction{7}{\rel{CvFhCgvhCgvFhCgCh}}
\pfassumefinal{5}{V}
\pfcontradiction{6}{\rel{ChCgChNhCgChVChChNhCgChhChNhCgChVChVhhChNhCgChVChV}}
\pfassume{1}{H}
\pfassume{2}{n}
\pfassumefinal{3}{d}
\pfcontradiction{4}{\rel{nHggdHnHgnnHggnnHggdHnHgndgnnHggnnHggdHnHgnn}}
\pfassume{3}{D}
\pfassumefinal{4}{c}
\pfcontradiction{5}{\rel{DnHcHDHgnnHcHDHDnnHcHDHDnDH}}
\pfassumefinal{4}{C}
\pfcontradiction{5}{\rel{DHCnHgnnHggnnHgnHCnHgnnCnHCnHgnn}}
\pfassume{2}{N}
\pfassume{3}{d}
\pfassumefinal{4}{c}
\pfcontradiction{5}{\rel{NNgdHcNgdHggdHgdNNgdHcNgdHggNgdHcNgdNNgdHcd}}
\pfassumefinal{4}{C}
\pfcontradiction{5}{\rel{NgdHggdCgdNNgdCgNgdHggdCgdHggdCgdNd}}
\pfassume{3}{D}
\pfassume{4}{c}
\pfassumefinal{5}{m}
\pfcontradiction{6}{\rel{NcHcHDNmDHDcHcHDNmHcHDNcHccHD}}
\pfassume{5}{M}
\pfassumefinal{6}{v}
\pfcontradiction{7}{\rel{HcHDMvHcHDNcHHcHDNcHcv}}
\pfassumefinal{6}{V}
\pfcontradiction{7}{\rel{DcVcHDNcHcHDcVcHDHVcHDH}}
\pfassume{4}{C}
\pfassumefinal{5}{m}
\pfcontradiction{6}{\rel{DmgmDNmDNgmgmDNmDHDDmgmDNmDHHgmDNm}}
\pfassumefinal{5}{M}
\pfcontradiction{6}{\rel{HDHCMHCMHDMCDHCMHDMHDMCMHCMHDM}}
\end{proof}

%%%%%%%%%%%%%%%%%%%%%%%%%%%%%%%%%%%%%%%%%%%%%%%%%%%%%%%%%%%%%%%%%%%%%%%%%%%%%%%%

\bibliography{bib}{}
\bibliographystyle{amsplain}

\end{document}